\documentclass[10pt]{extarticle}
\usepackage{amsmath}
\usepackage[dvips]{graphicx}
\usepackage{textcomp}
\usepackage{amsbsy}
\usepackage{latexsym}
\usepackage[mathscr]{eucal}
\usepackage{amsfonts}
\usepackage{amssymb}
\usepackage[usenames]{color}
\usepackage{stmaryrd}
\usepackage{amsthm}
\usepackage{mathabx}

\usepackage[utf8]{inputenc}

\usepackage{framed}

\usepackage{enumerate}

\usepackage[T1]{fontenc}
\usepackage{lmodern}

\usepackage{algorithm,algorithmic}

\makeatletter
\renewcommand{\ALG@name}{Algorithm ADAPT.}
\renewcommand{\fnum@algorithm}{\fname@algorithm}
\makeatother

\usepackage{epsf}

\usepackage{fullpage}

\usepackage{stackrel}

\usepackage{hyperref}
\hypersetup{ hidelinks = true, } 

\usepackage{comment}

\def\RR{\rm \hbox{I\kern-.2em\hbox{R}}}
\def\NN{\rm \hbox{I\kern-.2em\hbox{N}}}
\def\ZZ{\rm {{\rm Z}\kern-.28em{\rm Z}}}
\def\CC{\rm \hbox{C\kern -.5em {\raise .32ex \hbox{$\scriptscriptstyle
|$}}\kern
-.22em{\raise .6ex \hbox{$\scriptscriptstyle |$}}\kern .4em}}

\def\<{\langle}
\def\>{\rangle}

\def\Chi{\raise .3ex
\hbox{\large $\chi$}} 
\def\lsima{\hbox{\kern -.6em\raisebox{-1ex}{$~\stackrel{\textstyle<}{\sim}~$}}\kern -.4em}
\def\lsim{\hbox{\kern -.2em\raisebox{-1ex}{$~\stackrel{\textstyle<}{\sim}~$}}\kern -.2em}
\def\[{\Bigl [}
\def\]{\Bigr ]}
\def\({\Bigl (}
\def\){\Bigr )}
\def\[{\Bigl [}
\def\]{\Bigr ]}
\def\({\Bigl (}
\def\){\Bigr )}

\newcommand{\be}{\begin{equation}}
\newcommand{\ee}{\end{equation}}
\newcommand{\beu}{\begin{equation*}}
\newcommand{\eeu}{\end{equation*}}
\newcommand{\bea}{$$ \begin{array}{lll}}
\newcommand{\eea}{\end{array} $$}
\newcommand{\bi}{\begin{itemize}}
\newcommand{\ei}{\end{itemize}}

\newtheorem{theorem}{Theorem}
\newtheorem{remark}{Remark}
\newtheorem{lemma}{Lemma}

\newtheorem{corollary}{Corollary}

\newcommand{\ms}{H}
\newcommand{\evquad}{\mathcal{Z}_{\ortopar}}

\newcommand{\evrank}{\mathcal{W}_\dom}
\newcommand{\evdub}{\mathcal{A}_{\ortopar,\dom}}
\newcommand{\evtrib}{\mathcal{I}_{\delta,\widetilde{\delta},\ortopar,\dom}}
\newcommand{\evN}{\mathcal{N}_{\widetilde{\delta}}}
\newcommand{\evB}{\mathcal{B}_{\delta,\ortopar}}
\newcommand{\evC}{\mathcal{C}_{\delta}}
\newcommand{\evD}{\mathcal{D}_{\ortopar}}
\newcommand{\alldom}{B}
\newcommand{\dom}{\Omega}
\newcommand{\auxbas}{\varphi}
\newcommand{\bou}{\eta}
\newcommand{\fun}{u}
\newcommand{\multipli}{\rho}
\newcommand{\setcross}{\dom_{CV}}
\newcommand{\costi}{\theta}

\newcommand{\ortopar}{\varepsilon}

\DeclareMathOperator*{\argmin}{argmin}

\newcommand{\conf}{\beta}

\begin{document}
\bibliographystyle{plain}
\title{Multivariate approximation of functions on irregular domains\\
by weighted least-squares methods 
}
\author{Giovanni Migliorati\thanks{Sorbonne Universit\'e, UPMC Univ Paris 06, CNRS, UMR 7598, Laboratoire Jacques-Louis Lions, 4, place Jussieu 75005, Paris, France. email: migliorati@ljll.math.upmc.fr}}
%
%
\date{\today}
\maketitle
\begin{abstract}
\noindent 
We propose and analyse numerical algorithms based on weighted least squares for the approximation of a bounded real-valued function on a general bounded domain $\dom \subset \mathbb{R}^d$. Given any $n$-dimensional approximation space $V_n \subset L^2(\dom)$, the analysis in \cite{CM2016} shows the existence of stable and optimally converging weighted least-squares estimators, using a number of function evaluations $m$ of the order $n \ln n$. When an $L^2(\dom)$-orthonormal basis of $V_n$ is available in analytic form, such estimators can be constructed using the algorithms described in \cite[Section 5]{CM2016}. If the basis also has product form, then these algorithms have computational complexity linear in $d$ and $m$. In this paper we show that, when $\dom$ is an irregular domain such that the analytic form of an $L^2(\dom)$-orthonormal basis is not available, stable and quasi-optimally weighted least-squares estimators can still be constructed from $V_n$, again with $m$ of the order $n \ln n$, but using a suitable surrogate basis of $V_n$ orthonormal in a discrete sense. The computational cost for the calculation of the surrogate basis depends on the Christoffel function of $\dom$ and $V_n$. Numerical results validating our analysis are presented. 
\end{abstract}
\section{Introduction and overview of the paper}
Approximating an unknown function from its pointwise evaluations is a classical problem in mathematics.
\emph{Interpolation} and \emph{least squares} are two approaches to such a problem, see \emph{e.g.}~\cite{D,GKKW}. 
In this paper, we develop and analyse numerical methods based on least squares for the approximation of a bounded function $\fun:\dom\to \mathbb{R}$ on a general bounded domain $\dom \subset \mathbb{R}^{d}$ in any dimension $d$, that can be a challenging task due to the \emph{curse of dimensionality}. 
Approximation takes place in $L^2(\dom,\mu)$, the space of square-integrable functions with respect to $\mu:=\mu(\dom)$, the uniform probability measure on $\dom$. 
Given a finite $n$-dimensional linear space $V_n\subset L^2(\dom,\mu)$, projection-type numerical methods select $u_n \in V_n$ that minimizes the approximation error of $\fun$ in $V_n$. 
Standard least squares are an example of such numerical methods, that construct $\fun_n$ from pointwise evaluations of $\fun$ at $m>n$ iid random samples from $\mu$.  
An important point in the analysis of least squares concerns how large $m$ has to be, compared to $n$, to ensure stability and good approximation properties of the estimator $\fun_n$.   

Recent works \cite{DH,JNZ,CM2016} have pointed out weighted least-squares methods as a well-promising approach for approximation in arbitrary dimension $d$.
In any domain $\dom\subseteq \mathbb{R}^d$ and with any finite-dimensional space $V_n\subset L^2(\dom,\mu)$, it was shown in \cite{CM2016} that weighted least-squares estimators $u_n \in V_n$ are stable and optimally converging in expectation, when the $m$ evaluations of $\fun$ are taken at iid random samples from a suitable probability measure $\sigma_n=\sigma_n(\dom)$ that depends on $V_n$ and $\mu$, and with $m$ being only linearly proportional to $n$ up to a logarithmic term, and independent of the ambient dimension $d$. 
This result is recalled in Theorem~\ref{thm:mio}.

For the computation of $u_n$ with the above guarantees, the analytic expression of an $L^2(\dom,\mu)$-orthonormal basis $(L_j)_{j\geq 1}$ is needed. 
If this is available, then one can generate the random samples from $\sigma_n$ and construct $\fun_n$ as described in \cite{CM2016}, and as recalled in Section~\ref{sec:analysis} as well. 
Moreover, if the orthonormal basis has product form, like \emph{e.g.}~when $\dom$ is a product domain, then the numerical methods developed in \cite{CM2016} generate random samples from $\sigma_n$ at a computational cost that scales linearly in both $d$ and $m$.  

In general, when $\dom$ is an irregular domain, the analytic expression of an $L^2(\dom,\mu)$-orthonormal basis is not known. Hence a suitable surrogate basis $\widetilde{L}_1,\ldots,\widetilde{L}_n$ of $V_n$ is needed, that replaces $L_1,\ldots,L_n$ and at the same time retains orthogonality with respect to some scalar product easy to evaluate on any domain $\dom$. 
In this setting, a convenient choice is to orthonormalize $\widetilde{L}_1,\ldots,\widetilde{L}_n$ using a discrete scalar product with iid random samples from $\mu$. 
We denote by $\widetilde{\fun}_n \in V_n$ the new weighted least-squares estimator of $\fun$ computed using the basis $\widetilde{L}_1,\ldots,\widetilde{L}_n$, that differs from $\fun_n$ whose computation uses $L_1,\ldots,L_n$. 
In the present paper we show that the estimator $\widetilde{\fun}_n$ can be constructed on general domains $\dom$ and that: 
\begin{itemize}
\item $\widetilde{\fun}_n$  is stable with high probability, quasi-optimally converging in expectation, and uses a number $m$ of evaluations of $\fun$ only linearly proportional to $n$ up to a logarithmic term, and independent of $d$;  
\item the numerical construction of $\widetilde{L}_1,\ldots,\widetilde{L}_n$ requires $\widetilde{m}$ iid random samples from $\mu$ and the QR factorisation of $\widetilde{m}$-by-$n$ matrices, where $\widetilde{m}$ scales as the $L^\infty(\dom)$ norm of the reciprocal of the Christoffel function of $V_n$ on $\dom$. Such a construction does not use any evaluation of $\fun$, nor does it require the knowledge of $L_1,\ldots,L_n$. 
\end{itemize}

The novel stability and convergence result for the estimator $\widetilde{\fun}_n$ are stated in Theorem~\ref{main_theo}, whose proof uses previous results from \cite{CDL,CM2016} and matrix Bernstein inequality.  
The convergence estimate reads as follows, where we use the $L^2(\dom,\mu)$ and $L^\infty(\dom)$ best approximation errors of $\fun$ in $V_n$, a parameter $\ortopar\geq 0$ related to the construction of $\widetilde{L}_1,\ldots,\widetilde{L}_n$, two unnamed constants $C_2,C_\infty>0$, and omit the technical details on the truncation of the estimator: 
\begin{equation}
\label{eq:simplif_res}
\mathbb{E}( \| \fun - \widetilde{\fun}_n  \|^2_{L^2(\dom,\mu)}) \leq \left(1 + \dfrac{  C_2 (1+\ortopar n) n }{m} \right)\inf_{v \in V_n} \| \fun - v \|^2_{L^2(\dom,\mu)}+\dfrac{ C_\infty (1+\ortopar n) n }{ \widetilde{m}} \inf_{v \in V_n} \| \fun - v \|^2_{L^\infty(\dom)} + \textrm{trunc}. 
\end{equation}
The parameters $m$ and $\widetilde{m}$ essentially scale linearly and superlinearly in $n$, respectively.

The term $\ortopar n$ arises from the missing discrete orthogonality of $\widetilde{L}_1,\ldots,\widetilde{L}_n$, that occurs in any orthonormalisation process due to numerical cancellation. When $\widetilde{L}_1,\ldots,\widetilde{L}_n$ are constructed by Householder QR factorisation, $\ortopar$ provably does not exceed $\epsilon_M \widetilde{m} n^{3/2}$ where $\epsilon_M\approx 10^{-16}$ is the machine precision of arithmetic calculations, and thus the term $\ortopar n$ is completely negligible for wide ranges of $n$ and $\widetilde{m}$. 

The construction of the estimator $\widetilde{\fun}_n$ uses $m$ evaluations of $\fun$ at iid random samples drawn from a surrogate discrete probability measure $\widetilde{\sigma}_n=\widetilde{\sigma}_n(\dom)$ that emulates $\sigma_n$, and that depends on $\widetilde{L}_1,\ldots,\widetilde{L}_n$ and $\mu$. 
The random samples from $\widetilde{\sigma}_n$ 
can be generated by subsampling the QR factorisation of a suitable matrix that depends on $\widetilde{L}_1,\ldots,\widetilde{L}_n$.
Similar applications of QR factorisation 
have been used for the computation of Fekete points \cite{AS}, and for the construction of randomised quadratures \cite{NZX}.    

In \cite{AH} a different method based on SVD truncation has been analysed, for the same purposes of function approximation on irregular domains. 
For that method, similar error estimates as \eqref{eq:simplif_res} have been obtained in \cite{AH}, but requiring a number $m$ of function evaluations that scales 
superlinearly 
in $n$, and using a different best approximation error that depends on the SVD truncation parameter. 

In \cite{AC}  
a method similar to 
our forthcoming 
Algorithm 1 
has been proposed, 
and its convergence 
in probability
has been analysed   
assuming exact discrete orthonormality of the surrogate basis, \emph{i.e.}~assuming $\varepsilon=0$.

The structure of our paper is the following: in Section~\ref{sec:analysis} we recall from \cite{CM2016} some results on approximation by weighted least-squares methods, and describe the additional challenges encountered when applying such methods to irregular domains. 
In Section~\ref{sec:main_res} we state Theorem~\ref{main_theo}. 
Its proof is postponed to Section~\ref{sec:intermediate}. 
In Section~\ref{sec:const_basis} we describe the construction of $\widetilde{L}_1,\ldots,\widetilde{L}_n$. 
In Section~\ref{sec:numerical_section} we propose two algorithms 
(Algorithm 1 and Algorithm 2) that compute the weighted least-squares estimator $\widetilde{u}_n$.
Section~\ref{sec:numerics} contains some numerical tests that validate our analysis. 
In Section~\ref{sec:conclu} we draw some conclusions. 

\section{Weighted least-squares approximation on irregular domains}
\label{sec:analysis}
\noindent
Given a bounded domain $\dom \subset \mathbb{R}^d$, we consider the problem of approximating a bounded function $\fun:\dom \to \mathbb{R}$ 
from its pointwise evaluations at independent random samples uniformly distributed over $\dom$. 
Without loss of generality we suppose that $\dom \subseteq \alldom :=[-1,1]^d$.  
Denote with $\mu=\mu(\dom)$ the uniform probability measure on $\dom$, and with $(L_j)_{j\geq 1}$ an orthonormal basis of $L^2(\dom,\mu)$, where the $L^2(\dom,\mu)$ norm is denoted as $\| u \|:=\sqrt{\langle u,u \rangle}$ and $\langle u,v \rangle:=\int_\dom uv \, d\mu$ for any $u,v \in L^2(\dom,\mu)$.
The Euclidean norm in $\mathbb{R}^n$ is indicated with $\ell^2$.  
For any $n\geq 1$, we denote by $V_n:=\textrm{span}(L_1,\ldots,L_n) \subset L^2(\dom,\mu)$ an $n$-dimensional approximation space, and assume that $V_n$ contains the constant functions. 
We define the $L^2(\dom,\mu)$-projection of $\fun$ on $V_n$ as 
\begin{equation}
P_n \fun:=\argmin_{v \in V_n}  \| \fun - v \|, 
\label{eq:def_ex_est}
\end{equation}
and denote by $e_n(\fun):=\| \fun - P_n \fun \|$ the $L^2(\dom,\mu)$ best approximation error of $\fun$ in $V_n$.  
We also denote by $e_n^\infty(\fun):=\inf_{v \in V_n} \| \fun - v  \|_{L^\infty(\dom)}$ the best approximation error in $L^\infty$. 
Using $L_1,\ldots,L_n$, we define the functions 
$$
k_n(y):=\sum_{j=1}^n | L_j(y)|^2, \quad w(y):=\dfrac{n}{k_n(y)}, \quad y \in \dom, 
$$
and the probability measure $\sigma_n$ on $\dom$ as
\begin{equation}
d\sigma_n:= w^{-1} d\mu = n^{-1} \sum_{j=1}^n L_j^2 d\mu.   
\label{eq:std_mes_opt}
\end{equation}
When $V_n$ is the total degree polynomial space, $k_n^{-1}$ is known as the Christoffel function, see \emph{e.g.}~\cite{F1986}. 
For any choice $y^1,\ldots,y^m \in \dom$ of $m$ points, we introduce the scalar product  
\begin{equation}
\langle u,v\rangle_m:= \frac1m \sum_{i=1}^m w(y^i) u(y^i) v(y^i), \qquad u,v \in L^2(\dom,\mu).   
\label{eq:scal_prod_first}
\end{equation}

In general the exact projection \eqref{eq:def_ex_est} cannot be computed. This motivates the interest in the discrete least-squares approach, where the $L^2(\dom,\mu)$ norm in \eqref{eq:def_ex_est} is replaced by the seminorm induced on $L^2(\dom,\mu)$ by the scalar product \eqref{eq:scal_prod_first}.  
Define the weighted least-squares estimator 
\begin{equation}
\fun_W^*:=\sum_{j=1}^n a_j L_j=\argmin_{v \in V_n}  \| \fun - v \|_{m},  
\label{eq:def_wls_est_ex}
\end{equation}
that can be computed from the minimal $\ell^2$-norm solution $a=(a_1,\ldots,a_n)^\top \in \mathbb{R}^n$ to the normal equations
\begin{equation}
G a = b, 
\label{eq:def_norm_eq_ex}
\end{equation}
where the Grammian matrix $G\in \mathbb{R}^{n \times n}$ is defined component-wise as $G_{jk}:=\langle L_j,L_k\rangle_m$ and the right-hand side $b\in \mathbb{R}^n$ has components $b_j=\langle L_j, \fun \rangle_m$. 
Throughout the paper, $I\in \mathbb{R}^{n\times n}$ denotes the identity matrix, and 
$$
\| A \| := \max_{ \| x \|_{\ell^2} =1} \| A x \|_{\ell^2}, 
\qquad 
\kappa(A):=\sqrt{\dfrac{ \lambda_{\max}(A^\top A) }{ \lambda_{\min}(A^\top A) } },
$$
denotes the spectral norm, respectively the condition number, of any matrix $A\in \mathbb{R}^{m\times n}$. 
For any $\delta \geq 0 $ define $\xi(\delta):=(1+\delta) \ln (1+\delta) - \delta >0$, that can be sandwiched as $(2 \ln(2)-1) \delta^2 \leq \xi(\delta)\leq \delta^2/2$ when $\delta \in [0,1]$.  
As in \cite{CM2016}, we suppose that $\fun\in L^2(\dom,\mu)$ satisfies a uniform bound with some known $\bou>0$: 
\begin{equation}
\| \fun \|_{L^\infty(\dom)} \leq \bou.
\label{eq:unif_bound}
\end{equation}
We then introduce the truncation operator $z \mapsto T_\bou(z) := \textrm{sign}(z) \min\{ |z|,\bou\}$, and define the truncated estimator 
$$
\fun_T^*:= T_\bou \circ \fun_W^*. 
$$

The following result was proven in \cite[Theorem 2.1 and Corollary 2.2]{CM2016}, in a slightly different form (here we rewrite it with $\alpha=2m^{-r}$, where $r>0$ is the same parameter as in \cite{CM2016}).      
\begin{theorem}
\label{thm:mio}
In any dimension $d$, for any domain $\dom\subseteq \mathbb{R}^d$ and any $\alpha,\delta \in (0,1)$, if 
$$
m \geq  \dfrac{n}{\xi(\delta)}  \ln\left( \dfrac{ 2n }{ \alpha} \right)
$$
and $y^1,\ldots,y^m \in \dom$ are $m$ iid random samples from $\sigma_n$ then 
$$
\Pr( \| G - I \| > \delta) \leq \alpha,  
$$ 
and if $\fun \in L^2(\dom,\mu)$ satisfies \eqref{eq:unif_bound} then the estimator $\fun_T^*$ satisfies  
$$
\mathbb{E}( \| \fun - \fun_T^*   \|^2 ) \leq \left( 1+ \dfrac{4 }{\xi(\delta) \ln(2n / \alpha) } \right) e_n(\fun)^2 +  4 \alpha  \bou^2.   
$$
\end{theorem}

The above result holds with any bounded or unbounded domain $\dom$ in any dimension, and in general approximation spaces $V_n$. In practice, the computation of the estimator $\fun_T^*$ requires the analytic expression of an $L^2(\dom,\mu)$-orthonormal basis $L_1,\ldots,L_n$, for the generation of the random samples from \eqref{eq:std_mes_opt} and for the construction of $G$ and $b$ in \eqref{eq:def_norm_eq_ex}. 
When $\dom=[-1,1]^d$, many $L^2(\dom,\mu)$-orthonormal basis can be constructed by tensorization, \emph{e.g.}~tensorized Legendre polynomials or tensorized wavelets. 
Other examples are available when $\dom$ has a symmetric structure, \emph{e.g.}~spherical harmonics on 
the sphere 
$\dom=\{ y\in \mathbb{R}^d \, : \, \| y  \|_{\ell^2}=1 \}$.

In general, when $\dom$ is an irregular domain, the analytic expression of the $L_j$ is not known.
This introduces additional challenges in the development and analysis of projection-type numerical methods for approximation on irregular domains.  
In principle, candidate replacements of $L_1,\ldots,L_n$ are functions $\widetilde{L}_1,\ldots,\widetilde{L}_n \in V_n$ not necessarily orthonormal in $L^2(\dom,\mu)$ that satisfy the following prescriptions: 
\begin{enumerate}
\item[\textbf{P1)}] $\widetilde{L}_1,\ldots,\widetilde{L}_n$ be orthonormal w.r.t.~a discrete scalar product, that can be easily evaluated with any domain $\dom$, in contrast to the $L^2(\dom,\mu)$ scalar product that requires integration over $\dom$;   
\item[\textbf{P2)}] $\textrm{span}(\widetilde{L}_1,\ldots,\widetilde{L}_n) = V_n$, since our goal is the approximation of $\fun$ in the space $V_n$.
\end{enumerate}

We now introduce some tools useful for the numerical construction of the basis $\widetilde{L}_1,\ldots,\widetilde{L}_n$. 
Let $\widetilde{y}^1,\ldots,\widetilde{y}^{\widetilde{m}} \in \dom$ be $\widetilde{m}$ iid random samples from $\mu$, and define the discrete scalar product  
\begin{equation}
\langle u,v \rangle_{\widetilde{m}}:= \frac{1}{\widetilde{m}} \sum_{i=1}^{\widetilde{m}} u(\widetilde{y}^i) v(\widetilde{y}^i), \qquad u,v \in L^2(\dom,\mu),     
\label{eq:disc_scal_prod}
\end{equation}
and $\| u \|_{\widetilde{m}}:=\sqrt{ \langle u,u \rangle_{\widetilde{m}}  }$. 
For any $\ortopar \geq0$, we say that $\widetilde{L}_1,\ldots,\widetilde{L}_n $ are $\ortopar$-orthonormal if 
\begin{equation}
\label{def_epsorto}
\sum_{j,k=1}^n |\langle \widetilde{L}_j , \widetilde{L}_k \rangle_{\widetilde{m}}-\delta_{jk}|^2 \leq \ortopar^2. 
\end{equation}
Orthonormalisation algorithms, \emph{e.g.}~Gram Schmidt-type or factorization-type algorithms, try to construct a set $\widetilde{L}_1,\ldots, \widetilde{L}_n \in V_n$ of functions orthonormal w.r.t~\eqref{eq:disc_scal_prod}, but they suffer from loss of orthogonality due to numerical cancellation.  
As a consequence, the $\widetilde{L}_1,\ldots, \widetilde{L}_n$ constructed by any such numerical method are only $\ortopar$-orthonormal for some $\ortopar > 0$, and prescription P1 can be fulfilled with the scalar product \eqref{eq:disc_scal_prod} up to a (hopefully small) loss of orthogonality quantified by \eqref{def_epsorto}.

Let us consider prescription P2. Define $\widetilde{V}_n:=\textrm{span}(\widetilde{L}_1,\ldots,\widetilde{L}_n)$ and denote with $\auxbas_1,\ldots,\auxbas_n \in V_n$ a collection of $n$ functions such that $\textrm{span}(\auxbas_1,\ldots,\auxbas_n)=V_n$.
These functions need not be orthonormal to a scalar product, but only linearly independent.  
Orthonormalisation algorithms construct each $\widetilde{L}_1,\ldots,\widetilde{L}_n$ as linear combinations of $\auxbas_1,\ldots,\auxbas_n$, ensuring $\widetilde{V}_n \subseteq V_n$.  
The coefficients of the linear combinations are computed from evaluations of $\auxbas_1,\ldots,\auxbas_n$ at $\widetilde{y}^1,\ldots,\widetilde{y}^{\widetilde{m}}$. 
Although linearly independent, the $\auxbas_i$ and $\auxbas_j$ with $i\neq j$ could be indistinguishable when evaluated at $\widetilde{y}^1,\ldots,\widetilde{y}^{\widetilde{m}}$, and when this happens the $\widetilde{L}_1,\ldots,\widetilde{L}_n$ do not span the whole $V_n$.  
Due to randomness in the $\widetilde{y}^1,\ldots,\widetilde{y}^{\widetilde{m}}$, in general spaces $V_n$ one can ensure P2 only with large probability. 
When $V_n$ is a polynomial space, in Section~\ref{sec:const_basis} we show that P2 can be ensured with probability one. 

For the time being we suppose that an $\ortopar$-orthonormal basis $\widetilde{L}_1,\ldots, \widetilde{L}_n$ is available for some $\ortopar>0$.  
A concrete algorithm for the construction of such a basis is described in Section~\ref{sec:const_basis}, together with suitable bounds for $\ortopar$.
Using $\widetilde{L}_1,\ldots,\widetilde{L}_n$, define the functions  
\begin{equation}
\widetilde{k}_n(y):=\sum_{j=1}^n | \widetilde{L}_j(y)|^2, \quad w(y):=\dfrac{\gamma}{\widetilde{k}_n(y)}, \quad y \in \dom, 
\label{eq:weight_other} 
\end{equation}
where $\gamma>0$ is a normalisation term defined later. 
Consider the set $\widetilde{\dom}:=\{\widetilde{y}^1,\ldots,\widetilde{y}^{\widetilde{m}}   \} \subset \dom$ containing $\widetilde{m}$ iid random samples from $\mu$, and define the discrete uniform probability measure $\widetilde{\mu}$ on $\widetilde{\dom}$ (\emph{i.e.}~$\widetilde{\mu}(\widetilde{y}^i)=\widetilde{m}^{-1}$ for all $i=1,\ldots,\widetilde{m}$) and the probability measure $\widetilde{\sigma}_n$ on $\widetilde{\dom}$ as  
\begin{equation}
d\widetilde{\sigma}_n:= w^{-1} d\widetilde{\mu} = \dfrac{1}{\gamma} \sum_{j=1}^n \widetilde{L}_j^2 d\widetilde{\mu},  
\label{eq:alt_mes_opt}
\end{equation}
with $\gamma:=\widetilde{m}^{-1} \sum_{i=1}^{\widetilde{m}} \sum_{j=1}^n (\widetilde{L}_j(\widetilde{y}^i))^2=\sum_{j=1}^n \langle \widetilde{L}_j, \widetilde{L}_j  \rangle_{\widetilde{m}}$.   

Let $y^1,\ldots,y^m \in \dom$ be $m$ iid random samples from $\widetilde{\sigma}_n$.   
Using these random samples and the scalar product \eqref{eq:scal_prod_first} with the weight $w$ chosen as in \eqref{eq:weight_other}, we define the Grammian matrix $\widetilde{G}\in\mathbb{R}^{n\times n}$ with components $\widetilde{G}_{jk}:=\langle \widetilde{L}_j,\widetilde{L}_k\rangle_{m}$, and the vector $\widetilde{b}\in \mathbb{R}^n$ with components $\widetilde{b}_j=\langle \widetilde{L}_j, \fun \rangle_m$. 
We now introduce the discrete projection $P_n^m$ on $\widetilde{V}_n$ and the weighted least-squares estimator $\fun_W$ as  
\begin{equation}
\fun_W:=P_n^m \fun :=\argmin_{v \in \widetilde{V}_n}  \| \fun - v \|_{m}.  
\label{eq:def_wls_est}
\end{equation}
The estimator $\fun_W$ can be computed by solving the normal equations
\begin{equation}
\widetilde{G} a = \widetilde{b},  
\label{eq:def_norm_eq}
\end{equation}
whose solution $a=(a_1,\ldots,a_n)^\top \in \mathbb{R}^n$ provides the coefficients of the expansion $\fun_W=\sum_{j=1}^n a_j \widetilde{L}_j$.   
Denote with $\fun_T$ the truncated estimator
$$
\fun_T:= T_\bou \circ \fun_W. 
$$

Define $\dom_{m} :=\overbrace{\dom\times \cdots \times \dom}^{ m \textrm{ times}}$. 
Throughout the paper, all the probability events belong to the Borel $\sigma$-algebra $\mathcal{B}(\dom_{m+\widetilde{m}})$ and $\Pr$ denotes the probability measure $(\otimes^{m} d\widetilde{\sigma}_n) \otimes
(\otimes^{\widetilde{m}} d\mu )$ on $\dom_{m+\widetilde{m}}$. 
The only exceptions are in Theorem~\ref{thm:mio} that uses $\mathcal{B}(\dom_{m})$ and $\Pr$ as $\otimes^{m} d\sigma_n$ on $\dom_m$, and in the forthcoming Theorem~\ref{theo_cdl} that uses $\mathcal{B}(\dom_{\widetilde{m}})$ and $\Pr$ as $\otimes^{\widetilde{m}} d\mu $ on $\dom_{\widetilde{m}}$. 

The following probability events are related to the construction of $\widetilde{L}_1,\ldots,\widetilde{L}_n$ satisfying P1 and P2:  
$$
\evquad  :=  \left\{ \widetilde{y}^1,\ldots,\widetilde{y}^{\widetilde{m}} \in \dom \, : \, \sum_{j,k=1}^n | \langle \widetilde{L}_j , \widetilde{L}_k\rangle_{\widetilde{m}} - \delta_{jk} |^2  \leq  \ortopar^2 \right\}, 
$$
$$
\evrank := \left\{ \widetilde{y}^1,\ldots,\widetilde{y}^{\widetilde{m}} \in \dom \, : \, \textrm{span}( \widetilde{L}_1,\ldots,\widetilde{L}_n) = V_n \right\}. 
$$
In the notation $\evrank$, the subscript points out the dependence on $\dom$ in the construction of $\widetilde{L}_1,\ldots,\widetilde{L}_n$, further discussed in Section~\ref{sec:const_basis}.
Notice that both events $\evquad $ and $\evrank$ do not depend on $y^1,\ldots,y^m$. 

We define the quantity
$$
K_n=K_n(\dom):=\sup_{y \in \dom} k_n(y),  
$$
that in general depends on $\dom$ and $V_n$.  
Thanks to the inclusion $L^\infty(\dom) \subset L^2(\dom,\mu)$ on any $\dom$ bounded, the lower bound $K_n\geq n$ holds for any $n\geq1$. 
When $\dom=[-1,1]^d$ and $V_n$ is a downward closed polynomial space, we also have the following upper bound from \cite{CCMNT2013} for any $n\geq 1$: 
\begin{equation}
\label{eq:up_bound_K}
K_n \leq n^2.
\end{equation}

For any $\widetilde{\delta} \in [0,1]$, define the probability event: 
$$
\evN:=\left\{ \widetilde{y}^1,\ldots,\widetilde{y}^{\widetilde{m}}\in \dom \textrm{ s.t. } \bigcap_{u \in V_n} \left\{ (1- \widetilde{\delta}) \| u \|^2 \leq \|u\|_{\widetilde{m}}^2  \leq (1+\widetilde{\delta}) \| u \|^2 \right\} \right\}.
$$
The next result was proven in \cite{CDL} in a slightly different form, that we rewrite with $\alpha=2m^{-r}$, as in Theorem~\ref{thm:mio}.  
\begin{theorem}
\label{theo_cdl}
In any dimension $d$, for any bounded $\dom \subset \mathbb{R}^d$, for any $\alpha>0$, $\widetilde{\delta} \in (0,1)$ and $n\geq1$, if  
\begin{equation*}
\widetilde{m}
\geq 
\dfrac{ K_n }{\xi(\widetilde{\delta})} \ln\left( \dfrac{2n}{\alpha} \right), 
\end{equation*}
and $\widetilde{y}_1,\ldots,\widetilde{y}_{\widetilde{m}}$ are iid random samples from $\mu$  
then $\Pr\left( \evN \right) > 1 - \alpha$.
\end{theorem}

It has been observed in \cite{AH} that the upper bound \eqref{eq:up_bound_K} can be generalised to bounded domains with the so-called $\lambda$-rectangle property, \emph{i.e.}~$\dom$ has the $\lambda$-rectangle property if $\exists \lambda \in (0,1)$ such that $\dom=\bigcup_{R \in \mathcal{R}} R$, where $\mathcal{R}$ is the set of (possibly overlapping) hyperrectangles $R\subseteq \dom$ such that $\inf_{R \in \mathcal{R}} \textrm{Vol}(R) = \lambda \textrm{Vol}(\dom)$. 

If the domain $\dom$ has the $\lambda$-rectangle property and $V_n$ is a  downward closed  polynomial space then 
\begin{equation}
\label{rect_prop}
K_n \leq \lambda^{-1} n^2, 
\end{equation}
see \cite[Theorem 6.6]{AH}. 
Simple domains that do not have the $\lambda$-rectangle property are \emph{e.g.}~the simplex and the ball.  
When $\dom$ is a convex or starlike domain and $V_n$ is a total degree polynomial space, asymptotic upper bounds for $K_n$ are available \emph{e.g.}~in \cite{Bos,KL,K,Xu}, 
see also \cite{PU} for estimates of $K_n$ when $d=2$. 
With more general domains $\dom$ and/or approximation spaces $V_n$, finding upper bounds for $K_n(\dom)$ is an open problem.

\subsection{Main results}
\label{sec:main_res}
This section contains Theorem~\ref{main_theo} and the analysis of a numerical algorithm that constructs $\widetilde{L}_1,\ldots,\widetilde{L}_n$.
Theorem~\ref{main_theo} states conditions ensuring that with large probability $\widetilde{G}$ stays close to the identity matrix in spectral norm, and that the estimator $\fun_T$ quasi-optimally converges in expectation, when the $\widetilde{L}_j$ are $\ortopar$-orthonormal. 
Theorem~\ref{main_theo} applies in general to any orthonormalisation algorithm. 
Its proof is postponed to Section~\ref{sec:intermediate}. 
In Theorem~\ref{main_theo} we assume that $\Pr(  \evquad \cap \evrank ) \geq 1-\conf$ for some $\conf \in [0,\frac12)$.  
This assumption means that, with probability at least $1-\conf$, the chosen orthonormalisation algorithm can construct $\widetilde{L}_1,\ldots,\widetilde{L}_n$ that are $\ortopar$-orthonormal and span the whole $V_n$, using $\widetilde{m}$ random samples $\widetilde{y}^1,\ldots,\widetilde{y}^{\widetilde{m}}$.    
In this respect, $\conf=\conf(\ortopar,\dom,\widetilde{m})$ represents the failure probability of the orthonormalisation algorithm. 
In some settings $\conf$ is known from the analysis, see Section~\ref{sec:const_basis}, and if not, in any case, it can be numerically estimated for the given domain $\dom$ and threshold $\ortopar$.

In Section~\ref{sec:const_basis} we discuss an orthonormalisation algorithm based on Householder QR factorisation, which constructs $\widetilde{L}_1,\ldots,\widetilde{L}_n\in V_n$ provably $\ortopar$-orthonormal with $\ortopar \approx \epsilon_M \widetilde{m}n^{3/2}$, and achieves $\conf=0$ when $V_n$ is a multivariate polynomial space. 
Corollary~\ref{coro_theo_main} contains the application of Theorem~\ref{main_theo} to such an algorithm. 

\begin{theorem}
\label{main_theo}
In any dimension $d$, for any bounded domain $\dom \subset \alldom$, for any $\alpha,\conf\in [0,\frac12)$, $\ortopar \in [0,1)$, $\delta \in (0,1-\ortopar)$, $\widetilde{\delta} \in (0,1)$ and $n\geq 1$, if the following conditions hold true
\begin{enumerate}[i)]
\item \label{itemthird} $m \geq \dfrac{4n(1+\ortopar)}{\delta^2} \ln\left( \dfrac{2n}{ \alpha} \right)$,
\item \label{itemfirst} $\widetilde{m}\geq \dfrac{K_n }{\xi(\widetilde{\delta})} \ln\left( \dfrac{2n}{\alpha} \right)$,
\item \label{itemsec} $\widetilde{y}^1,\ldots,\widetilde{y}^{\widetilde{m}} \stackrel{\textrm{iid}}{\sim} \mu$,
\item \label{itemfourth} $y^1,\ldots,y^m \stackrel{\textrm{iid}}{\sim} \widetilde{\sigma}_n$,
\item \label{itemfifth} $\Pr\left( \evquad \cap \evrank\right) \geq 1- \conf$, 
\end{enumerate}
then 
\begin{enumerate}[I)]
\item the matrix $\widetilde{G}$ satisfies 
\begin{equation}
\label{eq:first_result}
\Pr\left(
\| \widetilde{G} - I   \| \geq \delta + 
\ortopar
\right) 
\leq \alpha + \conf;
\end{equation}

\item if $\fun\in L^2(\dom,\mu)$ satisfies \eqref{eq:unif_bound} then the estimator $\fun_T$ satisfies 
\begin{align}
\mathbb{E}\left( \| \fun - \fun_T  \|^2   \right) 
\leq  
\left( 1+ 
\tau_2(n)
\right) e_n(\fun)^2  
+ 
\tau_\infty(n) 
 e_n^\infty(\fun)^2  
+ 8 \bou^2 (\alpha 
+ \conf ), 
\label{eq:second_result}
\end{align}
where 
$$
\tau_2(n):= 
\dfrac{1+\ortopar(n+1)}{1-\widetilde{\delta}}
\dfrac{ \delta^2 (1+\ortopar) }{ 4 
(1-\delta-\ortopar)^{2}
 \ln(2n/\alpha) }, 
\qquad 
\tau_\infty(n):= 
\dfrac{1+\ortopar(n+1)}{1-\widetilde{\delta}}
\dfrac{ \xi(\widetilde{\delta}) (1+\ortopar) n }{  
(1-\delta-\ortopar)^{2}
K_n \ln(2n/\alpha ) }. 
$$
\end{enumerate}
\end{theorem}

\begin{remark}[Comparison with Theorem~\ref{thm:mio}]
\label{stab_accu}
Theorem~\ref{thm:mio} and Theorem~\ref{main_theo} prove that $G$ and $\widetilde{G}$ are well-conditioned, respectively, when $m$ is of the order $n \ln n$, but with differently distributed random samples. 
In the proof of \eqref{eq:first_result}, $\widetilde{m}$ does not need to satisfy \ref{itemfirst}), and only needs to ensure 
a large probability 
of the event $\evquad \cap \evrank$ in \ref{itemfifth}).
Condition \ref{itemfirst}) is needed for the proof of \eqref{eq:second_result}.

The convergence estimates in Theorem~\ref{thm:mio} and Theorem~\ref{main_theo} differ due to term $\ortopar n$, whose presence is discussed in Remark~\ref{missing_orto_remma}, 
and due to the $L^\infty$-best approximation error, 
whose coefficient satisfies $\tau_\infty(n) \leq \tau_2(n)$ for any $n\geq 1$ such that $K_n \geq 2 n$. 
If $\widetilde{m}$ satisfies \ref{itemfirst}) with $K_n$ replaced by $\max\{K_n, n^2\}$, then $\tau_\infty$ decays to zero as $\ortopar/\ln n$.
\end{remark}

\begin{remark}[{Missing orthogonality of the $\widetilde{L}_1,\ldots,\widetilde{L}_n$}]
\label{missing_orto_remma}
In the proof of \eqref{eq:second_result}, 
the additional term $n\ortopar$ in \eqref{eq:term_disc_proj} arises from the fact that $\widetilde{L}_1,\ldots,\widetilde{L}_n$ are only $\ortopar$-orthonormal with $\ortopar>0$.
The term $n \ortopar$ propagates to $\tau_2$ and $\tau_\infty$ in \eqref{eq:second_result}, and is harmless as long as $\ortopar$ remains small. 
This is the case for wide ranges of $n$ and $\widetilde{m}$ since $\ortopar$ provably does not exceed $\epsilon_M \widetilde{m}n^{3/2}$ and $\epsilon_M\approx 10^{-16}$, 
see Section~\ref{sec:const_basis}.  
For example, if $\widetilde{m}=10^6$ and $n=10^3$ then $\ortopar \approx 10^{-6}$.
The numerical tests in Section~\ref{sec:numerics} 
show that 
even lower values of $\ortopar$ can be taken, of the order $10^{-12}$.

If $\widetilde{L}_1,\ldots,\widetilde{L}_n$ are assumed $\ortopar$-orthonormal with $\ortopar=0$ then, by Parseval's identity, \eqref{eq:term_disc_proj} simplifies to 
$$
\| P_n^m g\|^2 \leq \dfrac{1}{1-\widetilde{\delta} } \, \|a \|^2_{\ell^2},
$$
and the same proof of item II) gives \eqref{eq:second_result} with 
$$
\tau_2(n):= 
\dfrac{ \delta^2}{ 4 
(1-\delta)^{2}
 \ln(2n/\alpha) }, 
\qquad 
\tau_\infty(n):= 
\dfrac{ \xi(\widetilde{\delta})
 \, n }{ 
(1-\delta)^{2}
K_n \ln(2n/\alpha ) }, 
$$
being strictly decreasing functions that tend to zero as $n\to +\infty$. Notice that $K_n\geq n$.   
\end{remark}

\subsection{Proofs and intermediate results}
\label{sec:intermediate}
Given two events $X,Y$ such that $\Pr(Y)>0$, we denote by $\Pr(X |Y ):=\Pr(X \cap Y  )/\Pr(Y)$ the conditional probability of $X$ given $Y$. 

\begin{proof}[Proof of item I) in Theorem~\ref{main_theo}]
For convenience we define the events $\evdub:=\evquad \cap \evrank$, $\evB:=\{ \| \widetilde{G} - I   \| < \delta + \ortopar\}$, $\evC:=\{\| \widetilde{G} - \mathbb{E}(\widetilde{G}) \| < \delta \}$ and $\evD:=\{\| \mathbb{E}(\widetilde{G}) - I \| \leq \ortopar\}$.   
The expectation is on the $y^1,\ldots,y^m$, for given $\widetilde{y}^1,\ldots,\widetilde{y}^{\widetilde{m}}$. 
Indeed 
$\| \widetilde{G} - I   \| \leq \| \widetilde{G} - \mathbb{E}(\widetilde{G}) \| + \| \mathbb{E}(\widetilde{G}) - I \|$
implies 
$\evC\cap \evD \subseteq \evB$, and hence 
\begin{equation}
\Pr(\evB |  \evdub )  \geq \Pr( \evC \cap \evD | \evdub).
\label{eq:first_step}
\end{equation}
Using in sequence the definition of $\widetilde{G}$, linearity of expectation, \ref{itemfourth}) and \eqref{eq:alt_mes_opt} we obtain 
\begin{equation}
\label{eq:step_unk}
\mathbb{E}(\widetilde{G}_{jk} ) 
= 
\dfrac{1}{m}
\sum_{i=1}^{m} 
\mathbb{E}\left(
 w(y^i) \widetilde{L}_j(y^i) \widetilde{L}_k(y^i) 
\right)
= \sum_{i=1}^{\widetilde{m}}  
w(\widetilde{y}^i) 
\widetilde{L}_j(\widetilde{y}^i) \widetilde{L}_k(\widetilde{y}^i) 
\widetilde{\sigma}_n(\widetilde{y}^i)
= \sum_{i=1}^{\widetilde{m}}  
\widetilde{L}_j(\widetilde{y}^i) \widetilde{L}_k(\widetilde{y}^i) 
\widetilde{\mu}(\widetilde{y}^i) 
=\langle \widetilde{L}_j,\widetilde{L}_k \rangle_{\widetilde{m}},   
\end{equation}
for any $j,k=1,\ldots,n$.  
On the event $ \evdub$ for any $n\geq 1$ and $\ortopar \in [0,1)$ we have 
\begin{align}
\| \mathbb{E}(\widetilde{G}) - I \|^2 
\leq  
 \| \mathbb{E}(\widetilde{G}) - I  \|_{F}^2                                      
=
 \sum_{j,k=1}^{n} \left| \langle \widetilde{L}_j, \widetilde{L}_k \rangle_{\widetilde{m}} - \delta_{jk} \right|^2  
\leq 
\ortopar^2.
\label{bound_exp_Gtilde_id}
\end{align}
As a consequence of the above bound 
\begin{equation}
\Pr(\evD | \evdub ) = 1.
\label{eq:fourth_step}
\end{equation}

From Lemma~\ref{bernstein}, under conditions \ref{itemthird}) and  \ref{itemfourth}) it holds that 
\begin{equation}
\Pr(\evC |  \evdub 	) >  1 - \alpha. 
\label{eq:second_step}
\end{equation}

Using \eqref{eq:second_step} and \eqref{eq:fourth_step}, since 
$\Pr( \evC^C \cup \evD^C |  \evdub ) \leq \Pr( \evC^C |  \evdub) + \Pr( \evD^C |  \evdub ) \leq \alpha$  
we obtain 
\begin{equation}
\Pr( \evC \cap \evD |  \evdub) = 1 - \Pr( \evC^C  \cup \evD^C |  \evdub) > 1 - \alpha. 
\label{eq:fifth_step}
\end{equation}

Finally using in sequence \eqref{eq:first_step}, \eqref{eq:fifth_step} and \ref{itemfifth}) gives 
\begin{align*}
\Pr(\evB) \geq \Pr(\evB |  \evdub ) \Pr( \evdub) \geq 
& 
\Pr(\evC \cap \evD |  \evdub ) \Pr(   \evdub ) > (1-\alpha)(1-\conf) \geq 1-\alpha-\conf. 
\end{align*}
\end{proof}

\begin{lemma}
\label{lemma_bound_gamma}
On the event $\evquad$ the following holds:   
\begin{equation}
\delta_{jk} - \ortopar \leq \langle \widetilde{L}_j , \widetilde{L}_k\rangle_{\widetilde{m}} \leq \delta_{jk} + \ortopar, \quad j,k=1,\ldots,n;
\label{eq:bounds_gamma_par}
\end{equation}
\begin{equation}
n (1 - \ortopar) \leq \gamma = \sum_{j=1}^n \| \widetilde{L}_j \|_{\widetilde{m}}^2 \leq n (1 + \ortopar). 
\label{eq:bounds_gamma}
\end{equation}
\end{lemma}
\begin{proof}
The expression on the right-hand side below is equivalent to \eqref{eq:bounds_gamma_par}:  
\begin{align*}
\sum_{j,k=1}^n
| \langle \widetilde{L}_j , \widetilde{L}_k \rangle_{\widetilde{m}} - \delta_{jk} |^2 
\leq 
\ortopar^2
& \implies
| \langle \widetilde{L}_j , \widetilde{L}_k \rangle_{\widetilde{m}} - \delta_{jk} |^2 
\leq \ortopar^2, \quad j=1,\ldots,n.  
\end{align*}
For the proof of \eqref{eq:bounds_gamma} take $k=j$ in \eqref{eq:bounds_gamma_par} and then sum $j$ from $1$ to $n$.  
\end{proof}

The following result from \cite{Troppintro} is a consequence of Bernstein inequality for self-adjoint matrices.  
\begin{theorem}
\label{theobern}
Let $A\in \mathbb{R}^{n\times n}$ be a fixed 
matrix. Construct a 
symmetric random matrix $\ms \in \mathbb{R}^{n\times n}$ that satisfies 
$$
\mathbb{E}(\ms) = A \quad  \textrm{ and } \quad \|\ms \| \leq \gamma < +\infty.
$$ 
Compute the per-sample second moment $m_2(\ms) =\| \mathbb{E}(\ms^\top \ms)\|$.
Form the matrix sampling estimator
$$
\overline{\ms}
:= \dfrac{1}{m} \sum_{i=1}^m \ms^i, \textrm{ where each $\ms^i$ is an independent copy of $\ms$.}
$$
Then for all $\delta \geq 0$ the estimator satisfies 
$$
\Pr\left( \| 
\overline{\ms}
- A \| \geq \delta \right) \leq 2n \exp\left( \dfrac{ - m \delta^2/2  }{m_2(\ms)+2\gamma \delta /3 }   \right).
$$
\end{theorem}

In the next lemma we apply Theorem~\ref{theobern} on the event $\evdub=\evquad \cap \evrank$ and with the fixed matrix $A=\mathbb{E}(\widetilde{G})$, where the expectation is taken over $y^1,\ldots,y^m$ for given $\widetilde{y}^1,\ldots,\widetilde{y}^{\widetilde{m}}$. 

\begin{lemma}
\label{bernstein}
For any $\alpha \in (0,1)$, $\ortopar \in [0,1)$ and $n\geq1$, under conditions \ref{itemthird}) and \ref{itemfourth}) it holds that 
$$
\Pr(\evC | \evdub) >  1 - \alpha. 
$$
\end{lemma}
\begin{proof}
We define the random matrix $\ms=\ms(y)$ whose components are 
$$
\ms_{jk}(y):=w(y) \widetilde{L}_j(y) \widetilde{L}_k(y), \quad j,k=1,\ldots,n, 
$$
and $y$ is distributed as $\widetilde{\sigma}_n$. 
Using \ref{itemfourth}), define $\ms^i=\ms(y^i)$ for $i=1,\ldots,m$ as $m$ copies of the random matrix $\ms$. Notice that, from \ref{itemfourth}), on the event $ \evdub$ the $\ms^1,\ldots,\ms^m$ are  mutually independent.  They also satisfy  
$$
\widetilde{G}=\dfrac{1}{m} \sum_{i=1}^m \ms^i.   
$$
From linearity of expectation, condition \ref{itemfourth}) and \eqref{eq:step_unk} we obtain 
$
\mathbb{E}(\ms_{jk})
=
\mathbb{E}(\widetilde{G}_{jk} )  
= 
\langle \widetilde{L}_j, \widetilde{L}_k \rangle_{\widetilde{m}}  
$.  
For any $n \geq1$ and $\ortopar\in [0,1)$, from \eqref{bound_exp_Gtilde_id} on the event $ \evdub$  we have $\| \mathbb{E}(\ms) - I \| = \| \mathbb{E}(\widetilde{G}) - I \| \leq \ortopar$, and this is equivalent to 
$$
1- \ortopar \leq \| \mathbb{E}(\ms)  \| \leq 1+ \ortopar.
$$

Notice that, from the expression of $w$ in \eqref{eq:weight_other},  
$$
( \ms^\top \ms  )_{pq} 
= w^2 \widetilde{L}_p \widetilde{L}_q   \sum_{k=1}^n  \widetilde{L}_k^2 
=  w \gamma \widetilde{L}_p \widetilde{L}_q 
= \gamma \ms_{pq}, \quad p,q=1,\ldots,n, 
$$
and therefore $\ms^\top \ms= \gamma \ms$. 
Define now $m_2(\ms):= \| \mathbb{E}(\ms^\top \ms ) \| = \gamma \| \mathbb{E}(\ms) \|$. 
Thanks to the previous bounds 
$$
m_2(\ms) \leq \gamma (1+ \ortopar). 
$$

Since $\ms$ is a rank-one matrix, 
\begin{align*}
\| \ms \|^2 = \| \ms \|^2_F                
            = \textrm{trace}( \ms^T \ms )  
            =  \textrm{trace}( \gamma \ms ) 
             = w \gamma  \sum_{p=1}^n  \widetilde{L}_{p}^2  
             = \gamma^2. 
\end{align*}

Finally, on the event $\evdub$, we apply Theorem~\ref{theobern} with the fixed matrix $\mathbb{E}(\widetilde{G})$. On the event $\evdub$ the parameter $\gamma$ satisfies the uniform bound \eqref{eq:bounds_gamma}, and we obtain 
\begin{align*}
\Pr\left( \{  \|  \widetilde{G} - \mathbb{E}(\widetilde{G})  \| \geq \delta \} \, | \,  
 \evdub \right)  
\leq 
2n \exp\left(  \dfrac{ - m \delta^2 /2 }{ 
n(1+ \ortopar)^2
+ \frac{2   n(1+\ortopar) \delta}{3} } \right).  
\end{align*}
If condition \ref{itemthird}) holds true, since
$$
m 
\geq 
\ln\left( \dfrac{2n}{\alpha} \right)
\dfrac{4n(1+\ortopar)}{\delta^2}
>
\ln\left( \dfrac{2n}{ \alpha} \right)
\dfrac{2}{\delta^2}  \left( 
n(1+\ortopar)^2 
+ \dfrac{2 n(1+\ortopar)\delta}{3} 
\right) 
$$
we obtain the thesis. 
\end{proof}
\begin{proof}[Proof of item II) in Theorem~\ref{main_theo}]
The proof of the error estimate proceeds in the same way as the analogous proof of \cite[Theorem~2.1]{CM2016}, with some differences due to the missing orthogonality of the $\widetilde{L}_k$. 

From Theorem~\ref{theo_cdl} under \ref{itemfirst}) it holds $\Pr\left( \evN \right) > 1 - \alpha$. 
Since $\Pr( \evN^C \cup \evdub^C ) \leq 
\Pr(  \evN^C ) + \Pr( \evdub^C  ) 
\leq 
\alpha + \conf $ 
we obtain 
\begin{equation}
\Pr( \evN \cap \evdub ) 
= 
1 - \Pr( \evN^C  \cup \evdub^C   ) > 1 - \alpha -\conf.
\label{eq:third_step}
\end{equation}

Define $\evtrib:= \evB \cap \evN \cap  \evdub$. 
Combining \eqref{eq:third_step} and item I) it  holds that $\Pr( \evtrib) > 1 - 2\alpha - 2\conf$. 
On the event $\evtrib^C$ it holds $\| \fun - \fun_T \| \leq \| \fun \| + \| \fun_T \| \leq 2\bou$. 
Since $|\fun(y) - \fun_T(y)| \leq |\fun(y) - \fun_W(y)|$ for all $y \in \dom$, we also have $\| \fun - \fun_T \| \leq \| \fun - \fun_W \|$. 
Denoting $g:=\fun - P_n \fun$, on the event $\evtrib$ it holds that 
\begin{equation}
\label{eq:bound_trunc_proj_pit}
\| \fun - \fun_T  \|^2 
\leq  
\| \fun - \fun_W  \|^2 
=
 \| \fun - P_n \fun \|^2 + \| P_n \fun - P_n^m \fun   \|^2
= 
 \| g \|^2 + \| P_n^m g \|^2,  
\end{equation}
where we have used that $g$ is orthogonal to $V_n$, that $\textrm{span}(\widetilde{L}_1,\ldots,\widetilde{L}_n)=V_n$, and that $P_n^m P_n \fun = P_n \fun$. 
We expand $P_n^m g = \sum_{j=1}^n a_j \widetilde{L}_j$ over the $\widetilde{L}_j$, with $a=(a_j)_{j=1,\ldots,n}$ being the solution to $\widetilde{G}a=\widetilde{h}$ and $\widetilde{h}:=( \langle g,\widetilde{L}_k \rangle_m  )_{k=1,\ldots,n}$.

Using in sequence the norm equivalence in the event $\evN$, Lemma~\ref{lemma_bound_gamma}, $2a_j a_k \leq a_j^2 + a_k^2$, we obtain 
\begin{align}
\| P_n^m g \|^2 
= &
\left\| \sum_{k=1}^n a_k \widetilde{L}_k \right\|^2 
\nonumber
\\
\leq & 
\dfrac{1}{1-\widetilde{\delta}}  \left\| \sum_{k=1}^n a_k \widetilde{L}_k \right\|^2_{\widetilde{m}}  
\nonumber
\\
= & \dfrac{1}{1-\widetilde{\delta}}  \dfrac{1}{\widetilde{m}}  \sum_{i=1}^{\widetilde{m}} \sum_{k=1}^n  \sum_{j=1}^n a_k a_j \widetilde{L}_k(\widetilde{y}^i)  \widetilde{L}_j(\widetilde{y}^i)   
\nonumber 
\\
\leq & 
\dfrac{1}{1-\widetilde{\delta}}  
\dfrac{1}{\widetilde{m}}  
\left(  
\sum_{k=1}^n a_k^2 
\sum_{i=1}^{\widetilde{m}}
|\widetilde{L}_k(\widetilde{y}^i)|^2
+ 
\left|
\sum_{j=1}^n 
\sum_{k=1\atop k \neq j}^n
a_j a_k 
\sum_{i=1}^{\widetilde{m}}
\widetilde{L}_j(\widetilde{y}^i) 
\widetilde{L}_k(\widetilde{y}^i) 
\right|
\right)
\nonumber
\\
\leq & 
\dfrac{1}{1-\widetilde{\delta}}  
\left(  
(1+\ortopar)
\sum_{k=1}^n a_k^2 
+ 
\ortopar
\sum_{j=1}^n 
\sum_{k=1\atop k \neq j}^n
\dfrac{ a_j^2 + a_k^2}{2}
\right)
\nonumber
\\
\leq & 
\dfrac{1+\ortopar(n+1)}{1-\widetilde{\delta}}  
\| a \|^2_{\ell^2}. 
\label{eq:term_disc_proj}
\end{align}
Thus replacing \eqref{eq:term_disc_proj} in \eqref{eq:bound_trunc_proj_pit} provides the bound 
$$
\| \fun - \fun_T \|^2 \leq e_n(\fun)^2 + 
\dfrac{1+\ortopar(n+1)}{1-\widetilde{\delta}}  
\| a \|^2_{\ell^2}. 
$$
On the event $\evtrib$ item I) gives 
$\| \widetilde{G} \| \geq 1-\delta - \ortopar \implies 
\| \widetilde{G}^{-1} \| \leq (1-\delta - \ortopar)^{-1} $.   
Since $a=\widetilde{G}^{-1}\widetilde{h}$ we have 
$$
 \| \fun - \fun_T \|^2   \leq e_n(\fun)^2 + 
\dfrac{1+\ortopar(n+1)}{1-\widetilde{\delta}}  
\dfrac{1 }{(1-\delta-\ortopar)^{2} } \sum_{k=1}^n   | \langle g,\widetilde{L}_k \rangle_m    |^2.  
$$
Taking the total expectation over $y^1,\ldots,y^m,\widetilde{y}^{1},\ldots,\widetilde{y}^{\widetilde{m}}$ and using  $\Pr(\evtrib^C)\leq  2(\alpha+\conf)$ gives 
\begin{align*}
\mathbb{E}\left( \| \fun - \fun_T \|^2  \right) \leq 
& 
\left(
e_n(\fun)^2 +
\dfrac{1+\ortopar(n+1)}{1-\widetilde{\delta}} 
\dfrac{1}{ (1-\delta-\ortopar)^{2} } \sum_{k=1}^n  \mathbb{E}( | \langle g,\widetilde{L}_k \rangle_m    |^2  ) 
\right) \Pr(\evtrib) 
+ 
4\bou^2
\Pr(\evtrib^C) 
\\
\leq &
e_n(\fun)^2 + 
\dfrac{1+\ortopar(n+1)}{1-\widetilde{\delta}} 
\dfrac{1}{ (1-\delta-\ortopar)^{2} }
\sum_{k=1}^n  \mathbb{E}( | \langle g,\widetilde{L}_k \rangle_m    |^2  ) 
+ 
8\bou^2 (\alpha+\conf). 
\end{align*}

Denote with $\mathbb{E}_{\widetilde{y}}$ the expectation over $\widetilde{y}^1,\ldots,\widetilde{y}^{\widetilde{m}}$ and with $\mathbb{E}_y$ the expectation over $y^1,\ldots,y^m$.
For the second term above, using the independence of the random samples we have  
\begin{align*}
\mathbb{E}\left( | \langle g,\widetilde{L}_k \rangle_m  |^2   \right) 
= &
\dfrac{1}{m^2}
\mathbb{E}_{\widetilde{y}}
\left(
\sum_{i=1}^m
\sum_{j=1}^m
\mathbb{E}_y
\left( w(y^i) w(y^j) g(y^i) g(y^j) \widetilde{L}_k(y^i) \widetilde{L}_k(y^j)     
\right) 
\right) 
\\
= &
\dfrac{1}{m^2}\mathbb{E}_{\widetilde{y}}
\left(
m(m-1)
\left|
\mathbb{E}_y
\left( w(y) g(y) \widetilde{L}_k(y)      
\right) 
\right|^2
+ m 
\mathbb{E}_y
\left( 
|w(y) g(y) \widetilde{L}_k(y)|^2      
\right) 
\right) 
\\
= &
\mathbb{E}_{\widetilde{y}}
\left(
\left( 
1-\dfrac{1}{m}
\right)
\left|
\dfrac{1}{\widetilde{m} }
\sum_{i=1}^{\widetilde{m}}
g(\widetilde{y}^i) \widetilde{L}_k(\widetilde{y}^i)
\right|^2
+ \dfrac{1}{m} 
\dfrac{1}{\widetilde{m} }
\sum_{i=1}^{\widetilde{m}}
w(\widetilde{y}^i) 
|g(\widetilde{y}^i) \widetilde{L}_k(\widetilde{y}^i)|^2      
\right) 
\\
= &
\left( 
1-\dfrac{1}{m}
\right)
\dfrac{1}{\widetilde{m}^2 }
\underbrace{
\mathbb{E}_{\widetilde{y}}
\left(
\left|
\sum_{i=1}^{\widetilde{m}}
g(\widetilde{y}^i) \widetilde{L}_k(\widetilde{y}^i)
\right|^2
\right)
}_{I}
+ \dfrac{1}{m} 
\dfrac{1}{\widetilde{m} }
\underbrace{
\mathbb{E}_{\widetilde{y}}
\left( 
\sum_{i=1}^{\widetilde{m}}
w(\widetilde{y}^i) 
|g(\widetilde{y}^i) \widetilde{L}_k(\widetilde{y}^i)|^2      
\right)
}_{II}
.  
\end{align*}

Summing term I over $k$ gives 
\begin{align*}
\sum_{k=1}^n \mathbb{E}_{\widetilde{y}}\left( \left| \sum_{i=1}^{\widetilde{m}} g(\widetilde{y}^i) \widetilde{L}_k(\widetilde{y}^i)\right|^2 \right) 
&= 
\underbrace{
\sum_{k=1}^n
\sum_{i=1}^{\widetilde{m}}
\sum_{j=1 \atop j\neq i}^{\widetilde{m}}
\mathbb{E}_{\widetilde{y}}
\left( 
g(\widetilde{y}^i) \widetilde{L}_k(\widetilde{y}^i)
g(\widetilde{y}^j) \widetilde{L}_k(\widetilde{y}^j)      
\right) 
}_{III}
+ 
\underbrace{
\sum_{k=1}^n
\sum_{i=1}^{\widetilde{m}}
\mathbb{E}_{\widetilde{y}}
\left( 
|g(\widetilde{y}^i) \widetilde{L}_k(\widetilde{y}^i)|^2      
\right).  
}_{IV}
\end{align*}

We now show that Term III is equal to zero.  
On the event $\evtrib$ for any $k=1,\ldots,n$ and any $\widetilde{y}^1,\ldots,\widetilde{y}^{\widetilde{m}}$ it holds that $\widetilde{L}_k \in V_n $, and therefore 
\begin{align*}
III
= &
\sum_{k=1}^n
\sum_{i=1}^{\widetilde{m}}
\sum_{j=1 \atop j\neq i}^{\widetilde{m}}
\mathbb{E}_{\widetilde{y}^{\ell} \, : \, 
\ell \in \{ 1,\ldots,\widetilde{m} \} \setminus \{i, j \}
}
\left(
\mathbb{E}_{\widetilde{y}^{i}} 
\left(
g(\widetilde{y}^i)
\underbrace{
\mathbb{E}_{ \widetilde{y}^{j} }
\left(
 \widetilde{L}_k(\widetilde{y}^i)
g(\widetilde{y}^j) \widetilde{L}_k(\widetilde{y}^j)      
\right)
}_{
=:\overline{L}_k(\widetilde{y}^{i})
}
\right)
\right)
\\
= &
\sum_{k=1}^n
\sum_{i=1}^{\widetilde{m}}
\sum_{j=1 \atop j\neq i}^{\widetilde{m}}
\mathbb{E}_{\widetilde{y}^{\ell} \, : \, 
\ell \in \{ 1,\ldots,\widetilde{m} \} \setminus \{i, j \}
}
\left( 
\int_{\dom}
g(\widetilde{y}^i) 
\overline{L}_k(\widetilde{y}^{i})
\, d\mu(\widetilde{y}^i)
\right),
\end{align*}
where the function 
$\overline{L}_k=\overline{L}_k(\widetilde{y}^i)=
\mathbb{E}_{\widetilde{y}^{j}}(\widetilde{L}_k(\widetilde{y}^i)
g(\widetilde{y}^j) \widetilde{L}_k(\widetilde{y}^j))
$
is obtained as an average over $\widetilde{y}^j$ of functions 
in $V_n$, \emph{i.e.}~$\widetilde{L}_k(\widetilde{y}^i)$, multiplied by 
real-valued random variables, \emph{i.e.}~$g(\widetilde{y}^j)\widetilde{L}_k(\widetilde{y}^j)$. 
Therefore $\overline{L}_k$ does not depend on $\widetilde{y}^j$ and $\overline{L}_k \in V_n$.
Hence for any $k=1,\ldots,n$ the integral in the last line vanishes because $\overline{L}_k$ is orthogonal to $g$.  

For term IV, from Lemma~\ref{lemma_bound_gamma} we obtain
\begin{align*}
IV
\leq
\| g\|_{L^\infty(\dom)}^2
 \sum_{k=1}^n
\mathbb{E}_{\widetilde{y}}
\left(
\sum_{i=1}^{\widetilde{m}}
 | \widetilde{L}_k(\widetilde{y}^i)|^2      
\right)
=
 \widetilde{m} 
\| g\|_{L^\infty(\dom)}^2
 \sum_{k=1}^n
\mathbb{E}_{\widetilde{y}}
\left(
 \| \widetilde{L}_k
\|^2_{\widetilde{m}}      
\right)
\leq 
n
 \widetilde{m} 
(1+\ortopar)
\| g\|_{L^\infty(\dom)}^2. 
\end{align*}
Summing term II over $k$ and using Lemma~\ref{lemma_bound_gamma} gives 
\begin{align*}
\sum_{k=1}^n
\mathbb{E}_{\widetilde{y}}
\left( 
\sum_{i=1}^{\widetilde{m}}
w(\widetilde{y}^i) 
|g(\widetilde{y}^i) \widetilde{L}_k(\widetilde{y}^i)|^2      
\right)
& =
\mathbb{E}_{\widetilde{y}}
\left( 
\sum_{i=1}^{\widetilde{m}}
w(\widetilde{y}^i) 
|g(\widetilde{y}^i)|^2 
\sum_{k=1}^n|
\widetilde{L}_k(\widetilde{y}^i)|^2      
\right)
\\
& =
\gamma 
\mathbb{E}_{\widetilde{y}}
\left( 
\sum_{i=1}^{\widetilde{m}}
|g(\widetilde{y}^i)|^2 
\right)
\\
& =
\gamma
\widetilde{m}
\int_{\dom}
|g(y)|^2 
\, d\mu(y)
\\
& \leq
n(1+\ortopar)
\widetilde{m}
\|g \|^2. 
\end{align*}
Finally 
$$
\sum_{k=1}^n
\mathbb{E}\left( | \langle g,\widetilde{L}_k \rangle_m  |^2   \right)  
\leq 
(1+\ortopar)
\left( 
\dfrac{
n
}{  \widetilde{m} }
 \|g \|_{L^\infty(\dom)}^2
+ \dfrac{n}{m}
 \|g \|^2
\right),  
$$
and combining with \ref{itemfirst}) and \ref{itemthird}) gives \eqref{eq:second_result}. 
\end{proof}

\subsection{Construction of $\widetilde{L}_1,\ldots,\widetilde{L}_n$ with QR factorisation}
\label{sec:const_basis}
In this section we use Householder QR factorisation (hereafter HQRf) for the construction of $\widetilde{L}_1,\ldots,\widetilde{L}_n$.   
Let $\auxbas_1,\ldots,\auxbas_n \in V_n $ be $n$ linearly independent functions. 
Using the $\widetilde{m}$ random samples in \eqref{eq:disc_scal_prod}, we introduce the matrix $W\in \mathbb{R}^{\widetilde{m}\times n}$ defined component-wise as $W_{jk}:=\auxbas_{k}(\widetilde{y}^j)$ for $j=1,\ldots,\widetilde{m}$ and $k=1,\ldots,n$.

Recall the following result on HQRf, see \emph{e.g.}~\cite[Theorem 4.24]{S}: if $W$ has full rank, then it can be written uniquely in the form $W=QR$, where the columns of $Q\in \mathbb{R}^{\widetilde{m}\times n}$ form an orthonormal basis of the column space of $W$, and $R\in \mathbb{R}^{n \times n}$ is an upper triangular matrix with positive diagonal elements.  
Hence we can take
\begin{equation}
\label{def_el_q}
\widetilde{L}_k(\widetilde{y}^j)=\sqrt{\widetilde{m}} Q_{jk}, \quad j=1,\ldots,\widetilde{m}, \quad k=1,\ldots,n, 
\end{equation} 
and the factor $\sqrt{\widetilde{m}}$ makes the $\widetilde{L}_k$ orthonormal with \eqref{eq:disc_scal_prod}, while the columns of $Q$ are orthonormal with the Euclidean scalar product in $\mathbb{R}^{\widetilde{m}}$. 
For any $k=1,\ldots,n$ the analytic expression of $\widetilde{L}_k$ is given as a linear combination of $\auxbas_1,\ldots,\auxbas_k$ by 
\begin{equation}
\label{eq:def_basis_inv_R}
\widetilde{L}_k(y) := \sum_{j=1}^k \ell_j^k \auxbas_j(y), \quad y \in \dom, 
\end{equation}    
where for any $k=1,\ldots,n$ the vector $(\ell_1^k,\ldots,\ell_n^k)^\top \in \mathbb{R}^n$ is the solution to the linear system 
\begin{equation}
\label{eq:def_basis_inv_R_sys}
R^\top (\ell_1^k,\ldots,\ell_n^k)^\top = e^k		 
\end{equation}    
and $(e^k)_{k=1,\ldots,n}$
is the standard basis of $\mathbb{R}^n$, \emph{i.e.}~$e^k:=(e^k_1,\ldots,e^k_n)^\top \in \mathbb{R}^n$ and 
$e^k_j:=\delta_{jk}$ for any $j,k=1,\ldots,n$. 

The result above shows that if $\textrm{rank}(W)=n$ then the $\widetilde{L}_1,\ldots,\widetilde{L}_n$ constructed by 
\eqref{eq:def_basis_inv_R} satisfy P2. 
Conversely, if $\textrm{rank}(W)<n$ then the linear system \eqref{eq:def_basis_inv_R_sys} is singular, and P2 does not hold. 
Depending on the space $V_n$ and on the localisation of the supports of $\auxbas_1,\ldots,\auxbas_n$, two situations can occur:  
\begin{itemize}
\item $\auxbas_1,\ldots,\auxbas_n$ are \emph{globally supported} functions  on $\dom$. 
When $V_n$ is a multivariate polynomial space $V_n=V_\Lambda:=\textrm{span}\{ y^\nu  \, : \, \nu \in \Lambda, y \in \dom \}$ supported on a downward closed index set $\Lambda \subset \mathbb{N}^d_0$ with $n=\#(\Lambda)$, one can choose $\auxbas_1,\ldots,\auxbas_n$ as the tensorized monomial basis. 
In one dimension, whenever more than $n$ over $\widetilde{m}$ samples are distinct, the Vandermonde matrix $W$ has full rank. The same holds in higher dimension, but requiring that at least $n$ over $\widetilde{m}$ samples do not fall on any polynomial surface supported on $\Lambda$. 
In both cases, the probability that $\textrm{rank}(W)<n$ is formally zero, and also completely negligible when considering the numerical rank of $W$, since from \ref{itemfirst}) $\widetilde{m}$ is of the order $K_n \ln n \geq n \ln n$.

\item $\auxbas_1,\ldots,\auxbas_n$ are \emph{locally supported} functions  on $\dom$. In this case, the matrix $W$ is rank deficient whenever $\exists j \in [1,\ldots,n] \, : \, \textrm{supp}(\auxbas_j) \cap \{\widetilde{y}^1,\ldots,\widetilde{y}^{\widetilde{m}}\} = \emptyset$. The probability of such events is not zero, and can be calculated as a function of the size of $\textrm{supp}(\auxbas_j)$.  
Moreover, it might not be small if some of the $\auxbas_j$ have very localized support and $d$ is large. 
\end{itemize}

We now show that $\widetilde{L}_1,\ldots,\widetilde{L}_n$ in \eqref{def_el_q} satisfy P1 with $\ortopar$ not exceeding $\epsilon_M \widetilde{m} n^{3/2}$, 
where $\epsilon_M\approx 10^{-16}$ is the machine precision. 
From  \eqref{def_el_q} we obtain 
\begin{align}
\label{eq:error_qr_intro_combine}
\sum_{j,k=1}^n 
|\langle \widetilde{L}_j , \widetilde{L}_k \rangle_{\widetilde{m}} - \delta_{jk}|^2
=
\sum_{j,k=1}^n \left| \dfrac{1}{\widetilde{m}} \sum_{i=1}^{\widetilde{m}} \widetilde{L}_j(\widetilde{y}^i)  \widetilde{L}_k(\widetilde{y}^i) - \delta_{jk} \right|^2
= 
\sum_{j,k=1}^n \left| \sum_{i=1}^{\widetilde{m}} Q_{ij}Q_{ik} - \delta_{jk} \right|^2
=   
\| Q^\top Q - I \|_F^2,     
\end{align}
showing that $\ortopar$-orthonormality of the $\widetilde{L}_1,\ldots,\widetilde{L}_n$ is related to the loss of orthogonality of the matrix $Q$ due to numerical cancellation. 
The right-hand side in \eqref{eq:error_qr_intro_combine} can be estimated using classical results on backward error analysis for HQRf, like \cite[Theorem 1.5]{S} or \cite[Theorem 19.4]{Higham}.  
Using such results (see \emph{e.g.}~\cite[page 266]{S}) upper bounds for the orthogonality error of $Q$ take the form 
\begin{align}
\label{eq:error_qr_intro}  
\| Q^\top Q - I \|_F \leq 2\sqrt{n} \varphi(n,\widetilde{m}) \epsilon_M,  
\end{align}
where $\varphi=\varphi(n,\widetilde{m})$ is a slowly growing function of $n$ and $\widetilde{m}$. 
In particular \cite[Theorem 19.4]{Higham} shows that $\varphi(n,\widetilde{m}) \epsilon_M = c n \widetilde{m} \epsilon_M ( 1 - c n \widetilde{m} \epsilon_M )^{-1}$ with $c$ being a small numerical constant depending on the floating-point arithmetic.  
Hence $\| Q^\top Q - I \|_F \lesssim \epsilon_M \widetilde{m} n^{3/2}$ from \eqref{eq:error_qr_intro}, and thanks to \eqref{eq:error_qr_intro_combine} the $\widetilde{L}_1,\ldots,\widetilde{L}_n$ constructed by \eqref{def_el_q}--\eqref{eq:def_basis_inv_R} are provably $\ortopar$-orthonormal with $\ortopar\approx \epsilon_M \widetilde{m} n^{3/2}$. 

We now discuss the robustness of the construction of $\widetilde{L}_k$ to ill-conditioning of $W$. 
The matrix $W$ can be ill-conditioned, depending on the chosen basis $\auxbas_1,\ldots,\auxbas_n$ for the given domain $\dom$.   
As a remarkable property of HQRf, the error bound \eqref{eq:error_qr_intro} does not depend on $\kappa(W)$, ensuring $\ortopar$-orthonormality of $\widetilde{L}_1,\ldots,\widetilde{L}_n$ from \eqref{def_el_q} despite the ill-conditioning of $W$. 
The matrix $R$ inherits the same ill-conditioning of $W$, because $Q^\top Q \approx I$ and therefore $\kappa(W) \approx \kappa(R)$. 
Nonetheless, the linear system with matrix $R^\top$ in \eqref{eq:def_basis_inv_R_sys} can be solved with high accuracy by forward substitution, see \cite{H89}. 
Hence both P1 and P2 can be ensured also when $W$ 
is ill-conditioned.

The following corollary of Theorem~\ref{main_theo} is an immediate consequence of the above results on QR factorisation.

\begin{corollary} 
\label{coro_theo_main}
Given $\auxbas_1,\ldots,\auxbas_n \in V_n$ linearly independent, 
and given $\widetilde{y}^1,\ldots,\widetilde{y}^{\widetilde{m}}$ as in Theorem~\ref{main_theo}, let $W\in\mathbb{R}^{\widetilde{m}\times n}$ be the matrix with components $W_{ij}=\auxbas_j(\widetilde{y}^i)$, and let $\widetilde{L}_1,\ldots,\widetilde{L}_n$ be constructed from $QR=W$, the Householder QR factorisation of $W$. 
Under the same assumptions of Theorem~\ref{main_theo} 
but with item \ref{itemfifth}) replaced by 
\begin{center}
v bis) $\Pr\left( \{ \textrm{rank}(W)=n  \} \cap \{ \| Q^\top Q - I \|_F \leq \ortopar \} \right) \geq 1- \conf$,
\end{center}
the conclusions of Theorem~\ref{main_theo} in item I) and  item II) hold true. 
\end{corollary}

For given $\widetilde{y}^1,\ldots,\widetilde{y}^{\widetilde{m}}$ and $\auxbas_1,\ldots,\auxbas_n$ the event 
$\{ \textrm{rank}(W)=n  \} \cap \{ \| Q^\top Q - I \|_F \leq \ortopar \}$ 
in 
Corollary~\ref{coro_theo_main} can be checked if true or false, and thus its probability $1-\conf$ can be numerically estimated from the matrices $W$ and $Q$.   
If $\ortopar\approx \epsilon_M \widetilde{m} n^{3/2}$ then the inclusion $\{ \textrm{rank}(W)=n  \}  \subseteq \evquad \cap \evrank$ holds, and it is sufficient to check only the rank of $W$.   
If $V_n$ is a multivariate polynomial space and $\ortopar\approx \epsilon_M m n^{3/2}$ then $\conf=0$.

Before closing the section, we discuss the choice of the functions $\auxbas_1,\ldots,\auxbas_n$, that plays an important role in the numerical stability of the algorithm. 
The components $\ell_1^k,\ldots,\ell_n^k$ of the solution to \eqref{eq:def_basis_inv_R_sys} satisfy  
\begin{equation}
\label{eq:upp_bound_coef}
\ell_j^k 
= 
(R^{-T})_{kj}
=
(R^{-1})_{jk}
, 
\qquad j,k=1,\ldots,n,
\end{equation}
and might attain large values, \emph{e.g.}~due to possible bad scaling of the diagonal elements of $R$. 
Large values of $\ell_1^k,\ldots,\ell_n^k$ in \eqref{eq:def_basis_inv_R} reflect a poor choice of $\auxbas_1,\ldots,\auxbas_n$ to represent $\widetilde{L}_1,\ldots,\widetilde{L}_n$ on the given domain $\dom$. 
Indeed $R$ can always be made sufficiently 
close to the identity matrix 
if $\auxbas_1,\ldots,\auxbas_n$ are chosen sufficiently close (in the $L^2(\dom,\mu)$ sense) to 
$L_1,\ldots,L_n$.   
Unfortunately $L_1,\ldots,L_n$ are unknown if $\dom$ is irregular. 
In absence of a priori information on $L_1,\ldots,L_n$, we now show how to ensure that the $|\ell_j^k|$ in \eqref{eq:def_basis_inv_R} are not too large, by adapting $\auxbas_1,\ldots,\auxbas_n$ to the given domain $\dom$. 
To this aim, in Section~\ref{sec:adapting} we propose an algorithm that first rescales each $\auxbas_j$ as $\widetilde{\auxbas}_j:= \multipli_{j,\dom} \auxbas_j$, where the factor $\multipli_{j,\dom}>0$ depends on the domain $\dom$, and then computes the HQR factorisation $\widetilde{Q}\widetilde{R}=\widetilde{W}$ of the matrix $\widetilde{W}\in \mathbb{R}^{\widetilde{m}\times n}$ with components $\widetilde{W}_{ij}=\widetilde{\auxbas}_j(\widetilde{y}^i)$. 
The crucial point is that the algorithm choses $\multipli_{j,\dom}$ in such a way that $\widetilde{R}$ has all unitary diagonal elements. 
Using $\widetilde{R}$ the $\widetilde{L}_1,\ldots,\widetilde{L}_n$ can be obtained as
\begin{equation}
\label{eq:def_basis_inv_R_tilde}
\widetilde{L}_k(y) := \sum_{j=1}^k \widetilde{\ell}_j^k \widetilde{\auxbas}_j(y), \quad y \in \dom, 
\end{equation}    
by solving the linear system 
\begin{equation}
\label{eq:def_basis_inv_R_sys_prec_tilde}
\widetilde{R}^\top (\widetilde{\ell}_1^k,\ldots,\widetilde{\ell}_n^k)^\top = e^k
\end{equation}    
with forward substitution for any $k=1,\ldots,n$. 
Denote by $N\in \mathbb{R}^{n\times n}$ the upper triangular part of $\widetilde{R}$, that is a nilpotent matrix of index $n$.  
Thanks to the structure of $\widetilde{R}=I+N$, 
using Neumann series we can write $\widetilde{R}^{-1}=I + \sum_{s=1}^{n-1} (-1)^s N^s$. 
From \eqref{eq:def_basis_inv_R_sys_prec_tilde}
it holds $\widetilde{\ell}_j^k
=(\widetilde{R}^{-1})_{jk}$
for all $j,k=1,\ldots,n$.      
Therefore the coefficients $\widetilde{\ell}_j^k$
satisfy the safer 
bounds 
\begin{equation}
\label{better_bound_stab}
\widetilde{\ell}_j^k = 1 \textrm{ if } j=k,  
\qquad 
\widetilde{\ell}_j^k = 0 \textrm{ if } j>k, 
\qquad  
\max_{j,k=1,\ldots, n 
\atop 
j<k
}
|\widetilde{\ell}_j^k|= 
 \left\|  \sum_{s=1}^{n-1} (-1)^s N^s \right\|_{\max},
\end{equation}
that do not depend on the scaling of the diagonal elements of $R$. 
In practice the right hand-side of \eqref{better_bound_stab} exhibits only a slow growth w.r.t.~$n$ 
thanks to the alternating sign in the summation and to $N$ being nilpotent.
Therefore $N^s$ has at most $(n-s)^2/2$ nonzero components for any $s=1,\ldots,n$. 
The algorithmic construction of the $\multipli_{j,\dom}$ is discussed in Section~\ref{sec:adapting}. It uses HQRf of suitable incremental updates of the matrix $W$.
Notice that each $\widetilde{\auxbas}_j$ is obtained by rescaling $\auxbas_j$, and therefore $\textrm{rank}(W)=\textrm{rank}(\widetilde{W})$. 

In Section~\ref{sec:numerical_section} we describe two numerical algorithms that compute the estimator $\fun_T$, and their implementation.  
Both algorithms obey to the theoretical guarantees of Corollary~\ref{coro_theo_main}. 
The difference between the two algorithms is in the computation of $\widetilde{L}_1,\ldots,\widetilde{L}_n$. 
The first algorithm computes $\widetilde{L}_1,\ldots,\widetilde{L}_n$ from \eqref{eq:def_basis_inv_R} by solving 
\eqref{eq:def_basis_inv_R_sys},
directly using any chosen $\auxbas_1,\ldots,\auxbas_n$.  
The second algorithm computes $\widetilde{L}_1,\ldots,\widetilde{L}_n$ from \eqref{eq:def_basis_inv_R_tilde} by solving \eqref{eq:def_basis_inv_R_sys_prec_tilde}, 
adapting the chosen $\auxbas_1,\ldots,\auxbas_n$ to the domain $\dom$.
Both algorithms rely on the HQRf of $\widetilde{m}$-by-$n$ matrices whose cost is proportional to $\widetilde{m}n^2$. The second algorithm is numerically more stable thanks to \eqref{better_bound_stab}, but also computationally more demanding.

\section{Description of the algorithms}
\label{sec:numerical_section}
This section describes the numerical algorithms and their implementation.  
We start by describing the first algorithm.  
Given the domain $\dom$, the function $\fun$, the space $V_n$, the linearly independent functions $\auxbas_1,\ldots,\auxbas_n \in V_n$, 
the threshold $\ortopar$ and the bound $\bou$, the main tasks for the approximation of $\fun$ by the weighted least-squares estimator $\fun_T$ are the following, in the same sequential order:  

\begin{framed}
\paragraph{Algorithm 1:} computes the estimator $\fun_T$ using the given $\auxbas_1,\ldots,\auxbas_n$.
\begin{enumerate}
[leftmargin=1cm]
\item[Step  1:] generate $\widetilde{m}$ random samples $\widetilde{y}^1,\ldots,\widetilde{y}^{\widetilde{m}} \stackrel{\textrm{iid}}{\sim} \mu$;
\item[Step 2:] construct the matrix $W\in \mathbb{R}^{\widetilde{m}\times n}$ with components $W_{jk}:=\auxbas_k(\widetilde{y}^{j})$;   
\item[Test 1:] \textbf{IF} $\textrm{rank}(W) < n$ \textbf{THEN} set $\fun_T \equiv 0$ and goto Step 9; \textbf{ELSE} continue; 
\item[Step 3:] rescale all the columns of $W$ such that $\| \auxbas_k \|_{\widetilde{m}}=1$ (and keep track of the scaling factors);  
\item[Step 4:] compute $QR=W$, the Householder QR factorisation of $W$; 
\item[Test 2:] \textbf{IF} 
$\| Q^\top Q - I  \|_F > \ortopar$
\textbf{THEN} set $\fun_T \equiv 0$ and goto Step 9; \textbf{ELSE} continue; 
\item[Step 5:] construct $\widetilde{L}_1,\ldots,\widetilde{L}_n$ from \eqref{eq:def_basis_inv_R} by solving the linear system 
\eqref{eq:def_basis_inv_R_sys}; 
\item[Step 6:] generate $m$ random samples $y^1,\ldots,y^{m} \stackrel{\textrm{iid}}{\sim} \widetilde{\sigma}_n$;
\item[Step 7:] evaluate $\fun(y^1),\ldots,\fun(y^m)$;  
\item[Step 8:] compute the estimator $\fun_W$ of $\fun$ by solving the normal equations and set $\fun_T=T_\bou \circ \fun_W$;  
\item[Step 9:] return $\fun_T$.  
\end{enumerate}
\end{framed}

The algorithms for the generation of the random samples at Steps 1 and 6 are presented in Section~\ref{sec:gen_samples}. 
The algorithm that computes the $\widetilde{L}_k$ at Steps 2, 3, 4 and 5 is discussed in Section~\ref{sec:const_basis}. 
The construction of the normal equations at Step 8 is described in Section~\ref{sec:computation_LS}. 
The main purpose of Test 1 and Test 2 is to avoid wasting computational resources at the following steps, and in particular at Step 7. 
We now discuss the failure probabilities of each test. 
The failure probability of Test 1 depends on the localisation properties of the supports of $\auxbas_1,\ldots,\auxbas_n$, as discussed in Section~\ref{sec:const_basis}. 
Whenever Test 1 fails, one can restart the algorithm from Step 1 with the same $\auxbas_1,\ldots,\auxbas_n$ or with a different choice.  
Concerning Test 2, the analysis of the orthogonality error in Section~\ref{sec:const_basis} shows that, if $\textrm{rank}(W)=n$ and $\ortopar\approx \epsilon_M\widetilde{m}n^{3/2}$, then the failure probability of Test 2 is zero. 
This condition is only sufficient: for example in all the numerical tests in Section~\ref{sec:numerics} 
the failure probability is zero  
with $\ortopar =10^{-12}$.

The second algorithm is the following Algorithm 2. It is similar to Algorithm 1, and the differences are in the computation of $\widetilde{L}_1,\ldots,\widetilde{L}_n$ at Steps 3, 4 and 5. The algorithm ADAPT at Step 3 
performs several orthonormalisation sweeps combined with suitable rescaling of the columns of $W$,     
as described in Section~\ref{sec:adapting}.   
At Step 8, the construction of the normal equations again follows Section~\ref{sec:computation_LS} but using the QR factorisation $\widetilde{Q}\widetilde{R}=\widetilde{W}$ of the matrix $\widetilde{W}$. 

\begin{framed}
\paragraph{Algorithm 2:} computes the estimator $\fun_T$ adapting the given $\auxbas_1,\ldots,\auxbas_n$ to $\dom$.
\begin{enumerate}
[leftmargin=1cm]
\item[Step  1:] generate $\widetilde{m}$ random samples $\widetilde{y}^1,\ldots,\widetilde{y}^{\widetilde{m}} \stackrel{\textrm{iid}}{\sim} \mu$;
\item[Step 2:] construct the matrix $W\in \mathbb{R}^{\widetilde{m}\times n}$ with components $W_{jk}:=\auxbas_k(\widetilde{y}^{j})$;   
\item[Test 1:] \textbf{IF} $\textrm{rank}(W) < n$ \textbf{THEN} set $\fun_T \equiv 0$ and goto Step 9; \textbf{ELSE} continue; 
\item[Step 3:] compute the matrix $\widetilde{W}=\textrm{ADAPT}(W)$;  
\item[Step 4:] compute $\widetilde{Q}\widetilde{R}=\widetilde{W}$, the Householder QR factorisation of $\widetilde{W}$; 
\item[Test 2:] \textbf{IF} 
$\| \widetilde{Q}^\top \widetilde{Q} - I  \|_F > \ortopar$
\textbf{THEN} set $\fun_T \equiv 0$ and goto Step 9; \textbf{ELSE} continue; 
\item[Step 5:] construct $\widetilde{L}_1,\ldots,\widetilde{L}_n$ from \eqref{eq:def_basis_inv_R_tilde} by solving the linear system \eqref{eq:def_basis_inv_R_sys_prec_tilde}; 
\item[Step 6:] generate $m$ random samples $y^1,\ldots,y^{m} \stackrel{\textrm{iid}}{\sim} \widetilde{\sigma}_n$;
\item[Step 7:] evaluate $\fun(y^1),\ldots,\fun(y^m)$;  
\item[Step 8:] compute the estimator $\fun_W$ of $\fun$ by solving the normal equations and set $\fun_T=T_\bou \circ \fun_W$;  
\item[Step 9:] return $\fun_T$.  
\end{enumerate}
\end{framed}

\subsection{Generation of the random samples}
\label{sec:gen_samples}
The following sampling algorithms can be used, see \emph{e.g.}~\cite{De}. 
Independent random samples from $\mu$ on $\dom\subseteq \alldom=[-1,1]^d$ can be generated by \emph{rejection sampling}. First step: draw iid random samples $\widetilde{y}^1,\widetilde{y}^{2},\ldots$ from $\mu(\alldom)$, the uniform probability measure on $\alldom$. Second step: accept any random sample $\widetilde{y}^i$ drawn at the first step as a random sample from $\mu(\dom)$ whenever $\widetilde{y}^i\in \dom$, and reject it otherwise. On average, the number of accepted random samples is proportional to $\lambda(\dom)/\lambda(\alldom)$, where $\lambda(\cdot)$ denotes the Lebesgue measure. When $\lambda(\dom)$ is small compared to $\lambda(\alldom)=2^d$, or when $d$ is large, the algorithm above suffers from the curse of dimensionality. 
For less general domains $\dom$, \emph{e.g.}~polytopes or convex bodies, 
alternative MCMC sampling algorithms like \emph{hit and run} or \emph{random walk} can be used.

Independent random samples $y^1,\ldots,y^m$ from the discrete distribution $\widetilde{\sigma}_n$ can be generated, for example, by \emph{inverse transform sampling}. In this case, the computational cost for drawing one sample from $\widetilde{\sigma}_n$ is $\mathcal{O}(\ln(\widetilde{m}))$ when using \emph{binary search}, or $\mathcal{O}(1)$ when using \emph{the alias method}, that however requires an additional cost for the preparation of the hash table.   

\subsection{Adapting $\auxbas_1,\ldots,\auxbas_n$ to the domain $\dom$}
\label{sec:adapting}
The algorithm ADAPT takes as input $W\in \mathbb{R}^{\widetilde{m}\times n}$ with components $W_{ij}=\auxbas_j(\widetilde{y}^i)$ and produces as output $\widetilde{W}\in\mathbb{R}^{\widetilde{m}\times n}$ with components $\widetilde{W}_{ij}=\widetilde{\auxbas}_j(\widetilde{y}^i)$ such that the matrix $\widetilde{R}$ in the Householder QR factorisation $\widetilde{Q}\widetilde{R}=\widetilde{W}$ of $\widetilde{W}$ has unitary diagonal elements. 
Each $\widetilde{\auxbas}_j$ is constructed as $\widetilde{\auxbas}_j=\multipli_{j,\dom} \, \auxbas_j$ rescaling $\auxbas_j$ by a factor $\multipli_{j,\dom}>0$ that depends on $\dom$. At the first iteration, with $j=1$, $\widetilde{W}$ is initialized as the first column of $W$ renormalized.  
At iteration $j=2,\ldots,n$, the algorithm creates an auxiliary matrix $Z\in \mathbb{R}^{\widetilde{m} \times j}$ by juxtaposition of $\widetilde{W}\in \mathbb{R}^{\widetilde{m} \times (j-1)}$ with the $j$th renormalised column of $W$. 
Then the QR factorisation of $Z$ is computed. 
Finally, the matrix $\widetilde{W}$ is updated again by juxtaposition of $\widetilde{W}$ with the $j$th column of $W$ but this time rescaled by an appropriately chosen factor that produces $\widetilde{R}_{jj}=1$ in the matrix $\widetilde{R}$ such that $\widetilde{Q}\widetilde{R}=\widetilde{W}$. 
Notice that the rescaling operation when multiplying $\auxbas_j$ by $\multipli_{j,\dom}$ corresponds to a simple renormalisation of $\auxbas_j$ in $\ell^2$ only when $j=1$, 
due to the additional term $| \widetilde{R}_{jj} |^{-1}$ when $j\geq 2$.    
For convenience, in the description of the algorithm we denote by $W(:,j)$ the $j$th column of $W$, and we denote by $[A|b]\in\mathbb{R}^{\widetilde{m}\times(k+1)}$ the juxtaposition of any matrix $A \in \mathbb{R}^{\widetilde{m}\times k}$ with any vector $b\in\mathbb{R}^{\widetilde{m}}$. 

\begin{algorithm}
\caption{Computes $\widetilde{W}$ such that $\widetilde{W} =\widetilde{Q}\widetilde{R}$ and $\widetilde{R}_{jj}=1$ for all $j=1,\ldots,n$.}
\begin{algorithmic} 
\REQUIRE $W$
\ENSURE $\widetilde{W}$
\STATE $\multipli_{1,\dom} \leftarrow  \| W(:,1) \|_{\ell^2(\mathbb{R}^{\widetilde{m}})}^{-1}$              
\STATE $\widetilde{W}  \leftarrow \multipli_{1,\dom} W(:,1)$      
\FOR{$j = 2,\ldots,n$}
\STATE $\multipli_{j,\dom} \leftarrow \| W(:,j) \|_{\ell^2(\mathbb{R}^{\widetilde{m}})}^{-1}$
\STATE $Z \leftarrow [ \widetilde{W} \, | \,  \multipli_{j,\dom}W(:,j) ]$
\STATE $[\widetilde{Q},\widetilde{R}]=qr(Z)$
\STATE $\multipli_{j,\dom} \leftarrow \multipli_{j,\dom}  | \widetilde{R}_{jj} |^{-1}$
\STATE $\widetilde{W} \leftarrow [\widetilde{W} \, | \, \multipli_{j,\dom}W(:,j)]$
\ENDFOR
\end{algorithmic}
\end{algorithm}

\subsection{Computation of the weighted least-squares estimator}
\label{sec:computation_LS}
The estimator $\fun_W$ can be calculated by solving the normal equations \eqref{eq:def_norm_eq}. 
The matrix $\widetilde{G}$ can be rewritten as $\widetilde{G}=D^\top D/m$, where $D\in \mathbb{R}^{m\times n}$ is a matrix obtained by subsampling and reweighting the rows of the matrix $Q$ introduced in Section~\ref{sec:const_basis}, as we now describe.  
After sampling the $y^1,\ldots,y^m$ among the $\widetilde{y}^1,\ldots,\widetilde{y}^{\widetilde{m}}$, we can build a deterministic function $\mathcal{S}:[1,\ldots,m]\to [1,\ldots,\widetilde{m}]$ such that $y^i=\widetilde{y}^{\mathcal{S}(i)}$ for any $i=1,\ldots,m$. 
Using the function $\mathcal{S}$ and \eqref{def_el_q} we can build $D$ as
\begin{align*}
D_{ij} 
= 
\sqrt{w(\widetilde{y}^{\mathcal{S}(i)})} \widetilde{L}_j(\widetilde{y}^{\mathcal{S}(i)}) 
= 
\sqrt{\dfrac{\sum_{\ell=1}^{\widetilde{m}} \sum_{k=1}^n Q^2_{\ell k}}{  \sum_{k=1}^n Q^2_{\mathcal{S}(i),k}}} Q_{\mathcal{S}(i),j}, 
\quad i=1,\ldots,m, \quad j=1,\ldots,n.
\end{align*}
The right-hand side $\widetilde{b}$ of \eqref{eq:def_norm_eq} can be calculated component-wise as  
\begin{align*}
\widetilde{b}_j
= 
\langle \widetilde{L}_j, \fun \rangle_{m} 
= &
\dfrac{1}{m} \sum_{i=1}^{m}  w( \widetilde{y}^{\mathcal{S}(i)} )  \widetilde{L}_j( \widetilde{y}^{\mathcal{S}(i)} ) 
 \fun(  \widetilde{y}^{\mathcal{S}(i)}  )
=
\dfrac{
\sum_{\ell=1}^{\widetilde{m}} \sum_{k=1}^n Q^2_{\ell k}
}{ m
\sqrt{\widetilde{m}}
}
 \sum_{i=1}^{m}
\dfrac{
Q_{\mathcal{S}(i),j}
\fun(  \widetilde{y}^{\mathcal{S}(i)})
}{ 
 \sum_{k=1}^n Q^2_{\mathcal{S}(i),k}
}
, \quad j=1,\ldots,n.  
\end{align*}

It is worth to mention that the random samples $y^1,\ldots,y^m$ in 
Theorem~\ref{main_theo} are drawn from $\widetilde{\sigma}_n$ \emph{with replacement}. 
This preserves independence, which is needed in the proof of Lemma~\ref{bernstein} when using Bernstein inequality.
As an alternative, one can draw $y^1,\ldots,y^m$ again from $\widetilde{\sigma}_n$ but \emph{without replacement}. 
The corresponding function $\mathcal{S}$ is injective, and this avoids multiple occurrences of the same row in the matrix $D$.
However the generated $y^1,\ldots,y^m$ are not independent anymore, and one cannot invoke Theorem~\ref{theobern}. 
Nevertheless, such an approach is interesting because random samples generated without replacement can better concentrate around their mean than those generated with replacement.   

\section{Numerical examples with polynomial spaces}
\label{sec:numerics}
In this section the weighted least-squares estimator $\fun_T$ of $\fun$ on $V_n$ is computed by Algorithm 2, as described in Section~\ref{sec:numerical_section}. 
The functions $\auxbas_1,\ldots,\auxbas_n$ are chosen as the tensorized monomial basis supported on the given polynomial space. 
When reporting the numerical results, we mainly focus on the stability of the estimator and on its approximation error. 
The stability is quantified by the condition number $\kappa(\widetilde{G})$, and from item I) of Theorem~\ref{main_theo}, $\| \widetilde{G} - I \| \leq \delta+\ortopar$ implies $\kappa(\widetilde{G})< (1+\widehat{\delta}+\ortopar)/(1-\widehat{\delta}-\ortopar)$.  
In all the numerical tests in this section, the $\widetilde{L}_1,\ldots,\widetilde{L}_n$ constructed by Householder QR factorisation are always $\ortopar$-orthonormal with values of $\ortopar$ less than $10^{-12}$. 

We now describe the numerical estimation of the error $\mathbb{E}(\|\fun - \fun_T \|)$ in Theorem~\ref{main_theo}. 
Denote with $\setcross \subset \dom $ a set of $m_{CV}$ iid random samples uniformly distributed on $\dom$, chosen once and for all.   
For any draw of $y^1,\ldots,y^m,\widetilde{y}^1,\ldots,\widetilde{y}^{\widetilde{m}} \in \dom$ the approximation error is estimated as
\begin{align}
\|\fun - \fun_T \| 
&\approx 
\|\fun - \fun_T \|_{CV}:= \sqrt{ \dfrac{1}{m_{CV}} \sum_{  y \in \setcross  } |\fun(y) - \fun_T(y) |^2 }.
\label{eq:one_run_mc}
\end{align}
The error in expectation is then estimated as a Monte Carlo average by
$$
\mathbb{E}(\|\fun - \fun_T \|) \approx \mathbb{E}_{MC}^r(\|\fun - \fun_T \|_{CV}),
$$ 
with the average $\mathbb{E}_{MC}^r$ being over $r$ independent draws of the random samples $y^1,\ldots,y^m,\widetilde{y}^1,\ldots,\widetilde{y}^{\widetilde{m}}$ from their respective distributions.  
In the following numerical tests we choose $m_{CV}=10^5$ and $r=100$. 

As illustrative examples in dimension $d=2$, we choose $\dom$ as a Swiss cheese set, \emph{i.e.}~a compact set with holes, or the Mandelbrot set, or the annular set. 
With all the aforementioned domains $\dom$, upper bounds for $K_n(\dom)$ are not known.
For the choice of $\widetilde{m}$, we define a parameter $\costi=\costi(n,\dom)$ depending on $\dom$ and $n$, and then take $\widetilde{m} = \lceil \costi n \ln n \rceil$. 
In all the numerical tests, choosing $m=\lceil 4 n \ln n \rceil$ and any $\costi \geq 1$ largely suffices to maintain the condition number safely bounded as $\kappa(\widetilde{G})\leq10$.
As discussed in Remark~\ref{stab_accu}, the choice of $\costi$ is important for the accuracy of $\fun_T$.
Unless otherwise specified, we empirically choose $\costi=200$. 

\subsection{Example with a smooth function on a domain with holes}
Define $\dom:= H \setminus \{ E_1 \cup E_2\} $, where $H:=\textrm{Conv}(S)$ is the convex hull of the point set 
$$
S:=\left\{ (-0.4,0.2)^\top, (-0.7,-0.7)^\top, (0.5,-0.3)^\top,(0.8,0.7)^\top, (0,0.7)^\top, (0,-0.6)^\top   \right\} \subset \alldom,
$$ 
$E_1$ is a standard ellipse centered in $(-0.2,-0.3)^\top$ with semiaxes of length $0.15$ and $0.15/\sqrt2$, and $E_2$ is a standard ellipse centered in $(0.2,0.2)^\top$ with semiaxes of length $0.2$ and $0.2/\sqrt2$. 
The geometry of $\dom$ is shown in Figure~\ref{fig:swiss_details}. 
We consider the function 
\begin{equation}
\label{fun_cheese}
\fun(y)= (1+ 0.2 \, y_1 + 0.1 \, y_2)^{-1}, \quad y=(y_1,y_2)^\top \in \dom \subset \mathbb{R}^2. 
\end{equation}

The space $V_n$ is chosen as the polynomial space supported on the index set $\Lambda=\Lambda_{TD}^{d,k}:=\{ \nu \in \mathbb{N}_0^d \, : \, \| \nu \|_{\ell^1} \leq k\}$, a.k.a.~the total degree polynomial space of order $k$, whose dimension equals $n=\dim(V_n)=\#(\Lambda_{TD}^{d,k})=\binom{d+k}{k}$. Figure~\ref{fig:swiss_error} shows the error $\mathbb{E}_{MC}^r(\|\fun - \fun_T \|_{CV})$ and condition number $\kappa(\widetilde{G})$ when $m=\lceil n\ln n\rceil$ or $m=\lceil 4n\ln n \rceil$, 
and $\widetilde{m} = \lceil 200 n \ln n \rceil$. The error decreases exponentially w.r.t.~$k$,  
and $\widetilde{G}$ remains well-conditioned even when choosing $m=\lceil n\ln n\rceil$. 
Figure~\ref{fig:swiss_details} shows one shot of the random samples $y^1,\ldots,y^m,\widetilde{y}^1,\ldots,\widetilde{y}^{\widetilde{m}}$ (left figure) and two realizations of the pointwise error $y\mapsto |\fun(y) - \fun_T(y)|$ for $y \in \dom$ (center and right figures).
\begin{figure}[htbp]   
\includegraphics[height=6.5 cm ,clip]{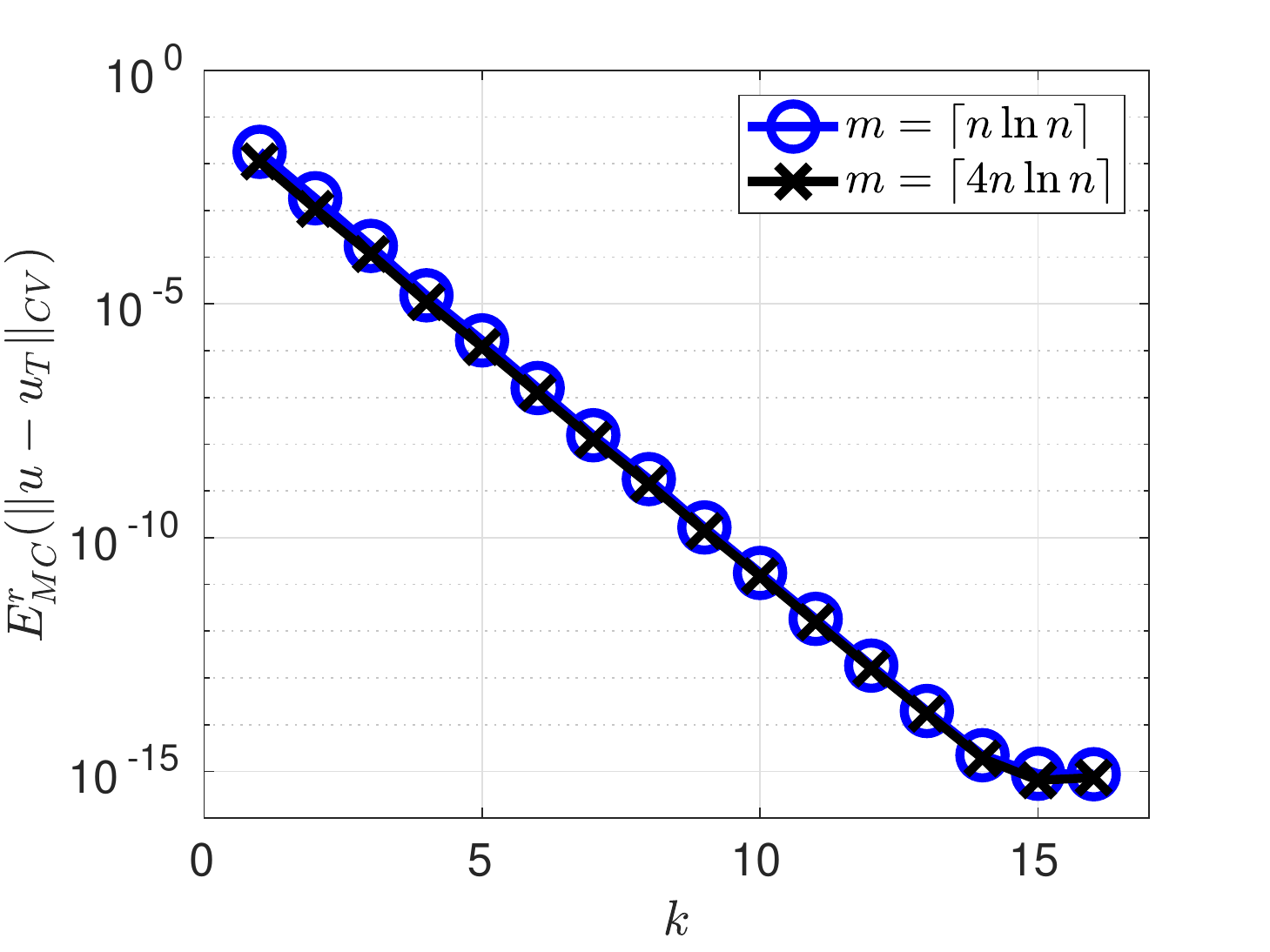}
\includegraphics[height=6.5 cm ,clip]{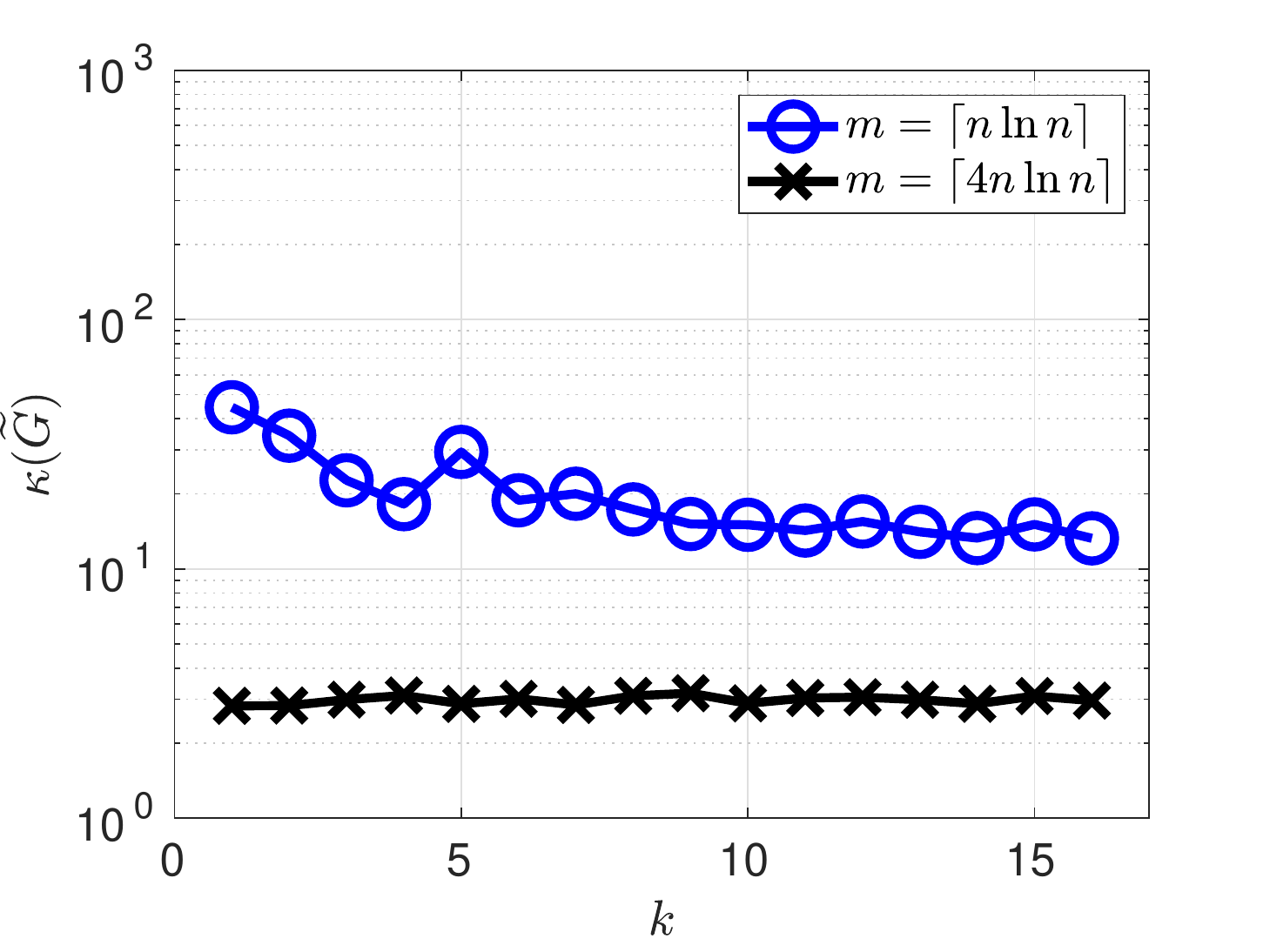}
\caption{\label{fig:swiss_error} 
Left: error $\mathbb{E}_{MC}^r(\|\fun - \fun_T \|_{CV})$ for the function \eqref{fun_cheese}. 
Right: condition number $\kappa(\widetilde{G})$.}
\end{figure}
\begin{figure}[htbp]   
\hspace{-1.0cm}
\includegraphics[height=5.1 cm ,clip]{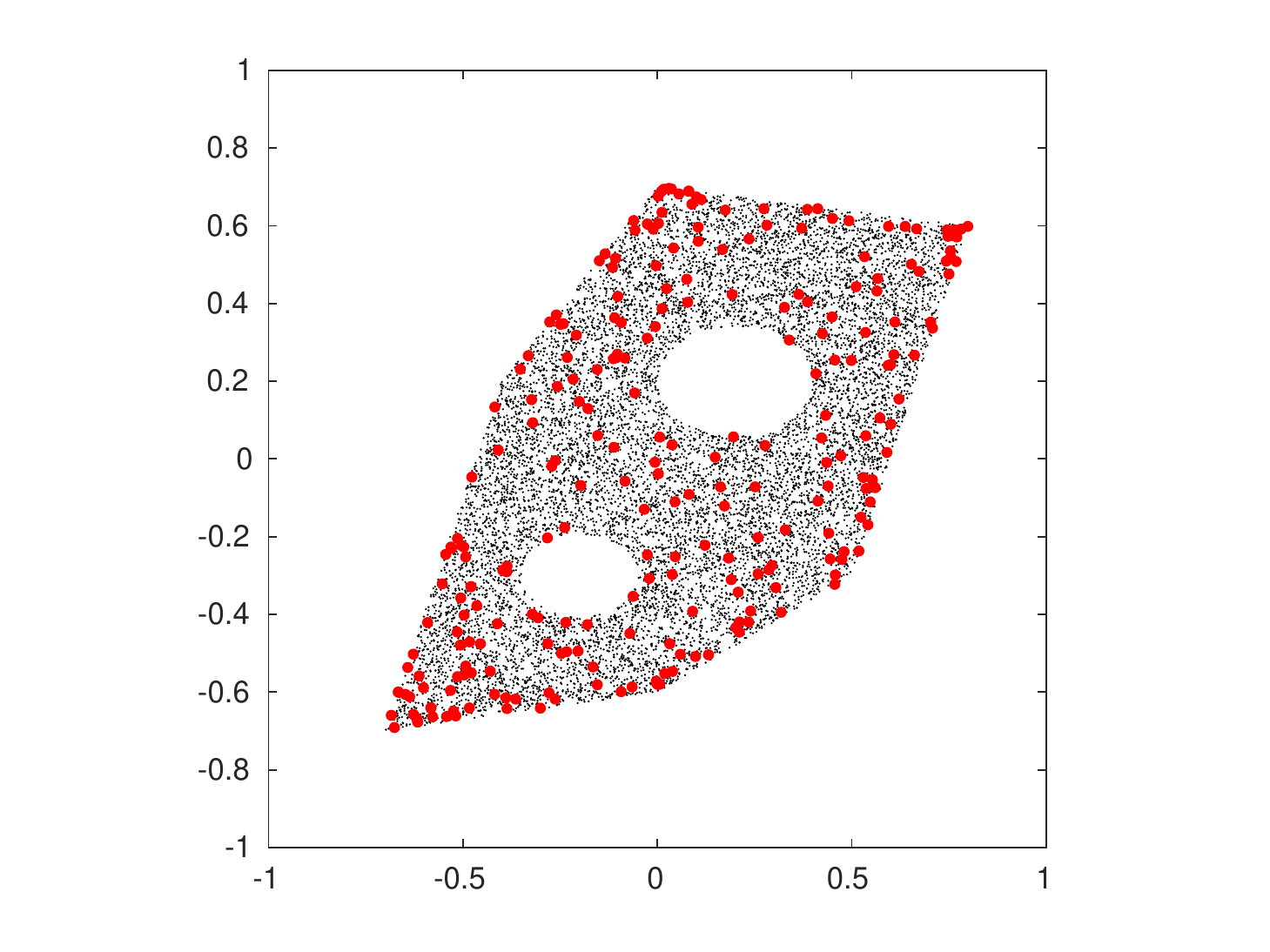}
\hspace{-1.3cm}
\includegraphics[height=5.1 cm ,clip]{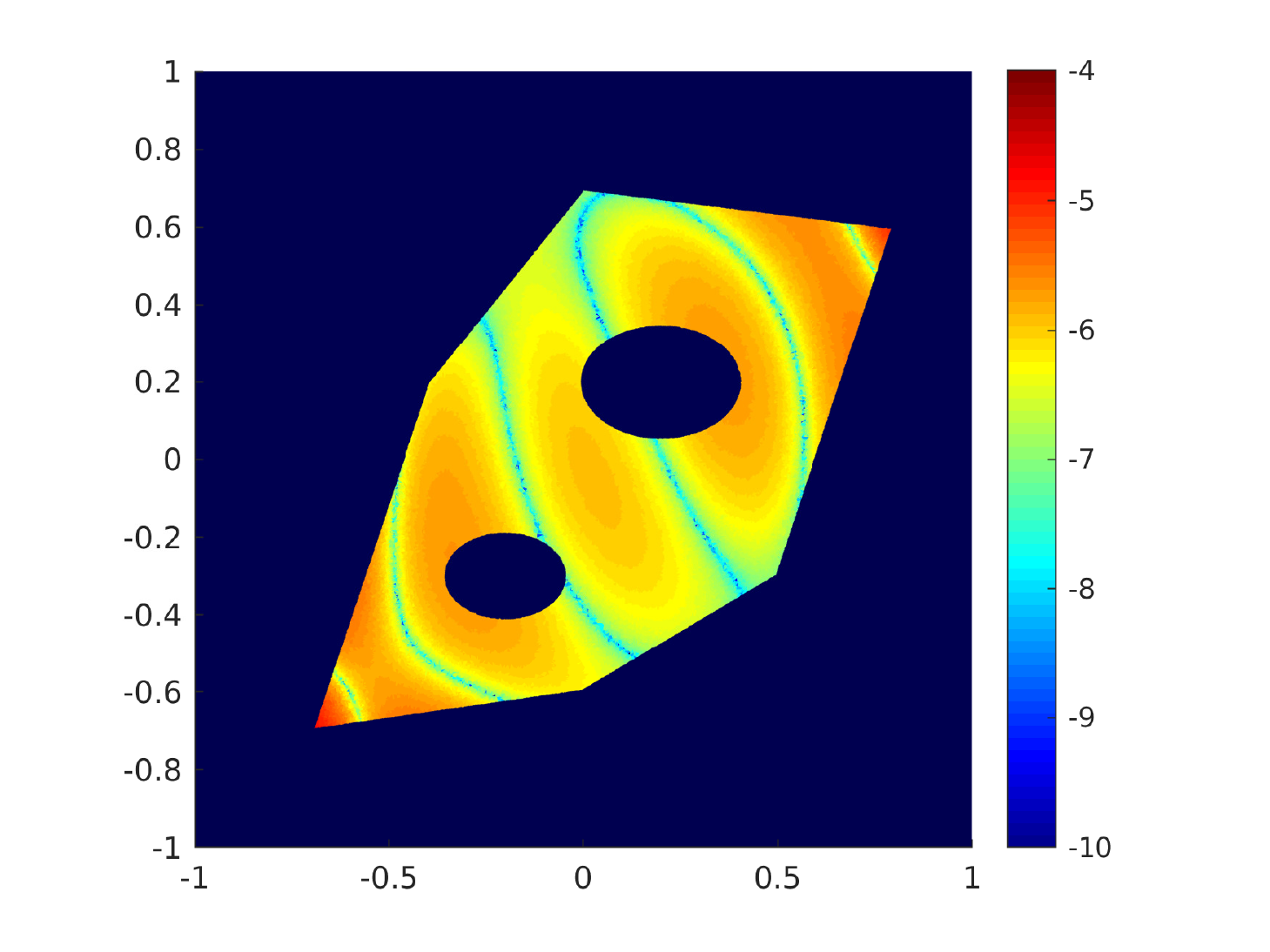}
\hspace{-1.0cm}
\includegraphics[height=5.1 cm ,clip]{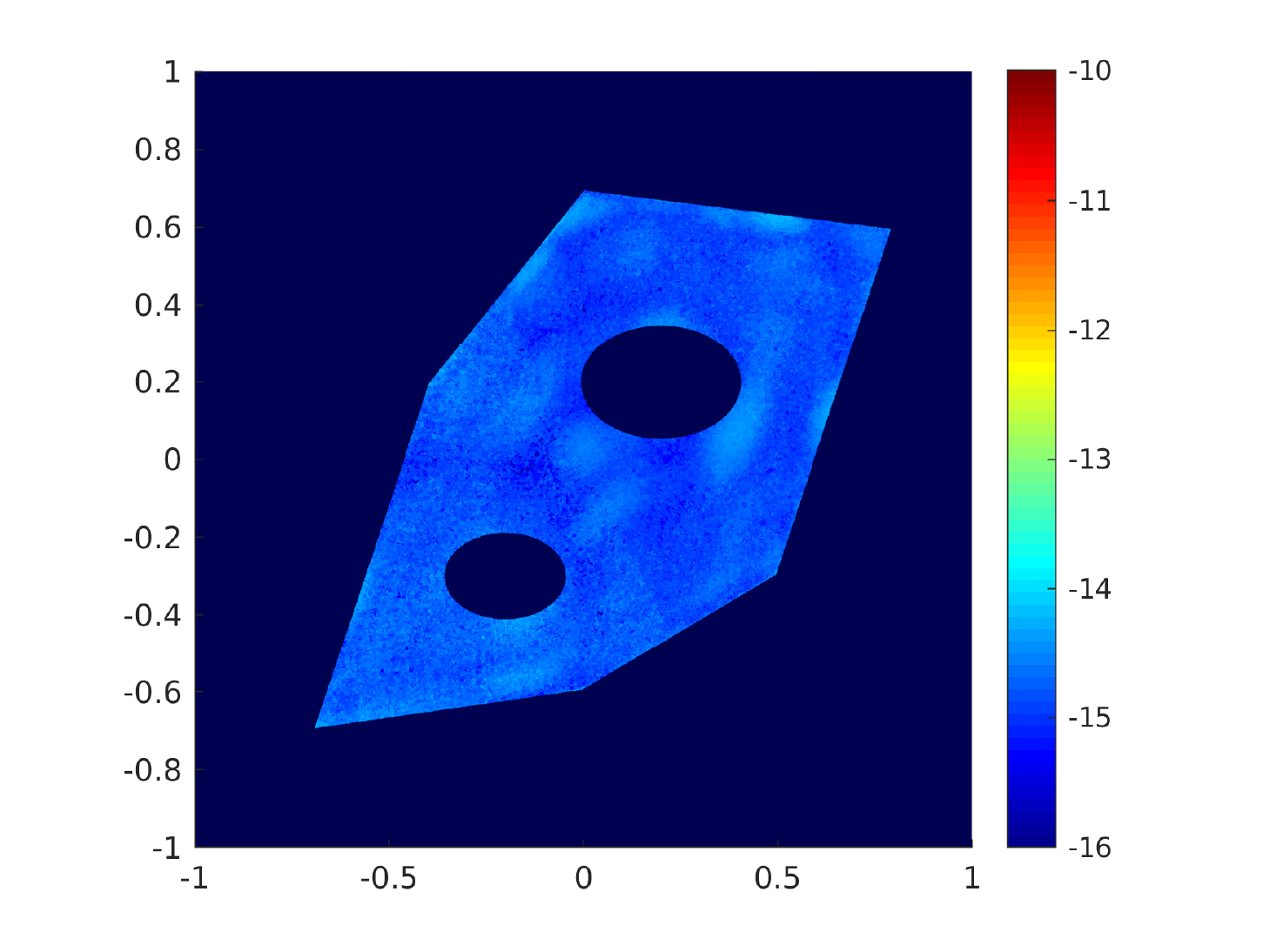}
\caption{\label{fig:swiss_details} 
Left: one realization of $y^1,\ldots,y^m$ (red dots) and $\widetilde{y}^1,\ldots,\widetilde{y}^{\widetilde{m}}$ (black dots) with $k=5$. 
Center: one realization of the error $y\mapsto \log_{10}|\fun(y) - \fun_T(y)|$ in $\dom$, with $k=5$. 
Right: one realization of the error $y\mapsto \log_{10}|\fun(y) - \fun_T(y)|$ in $\dom$, with $k=15$. 
All realizations are taken from the simulation in Figure~\ref{fig:swiss_error}, with $m=\lceil 4n\ln n \rceil$ and $\widetilde{m}=\lceil 200 n\ln n \rceil$. The dark blue region corresponds to $\alldom\setminus\dom$. 
}
\end{figure}
\begin{figure}[htbp]   
\includegraphics[height=6.5 cm ,clip]{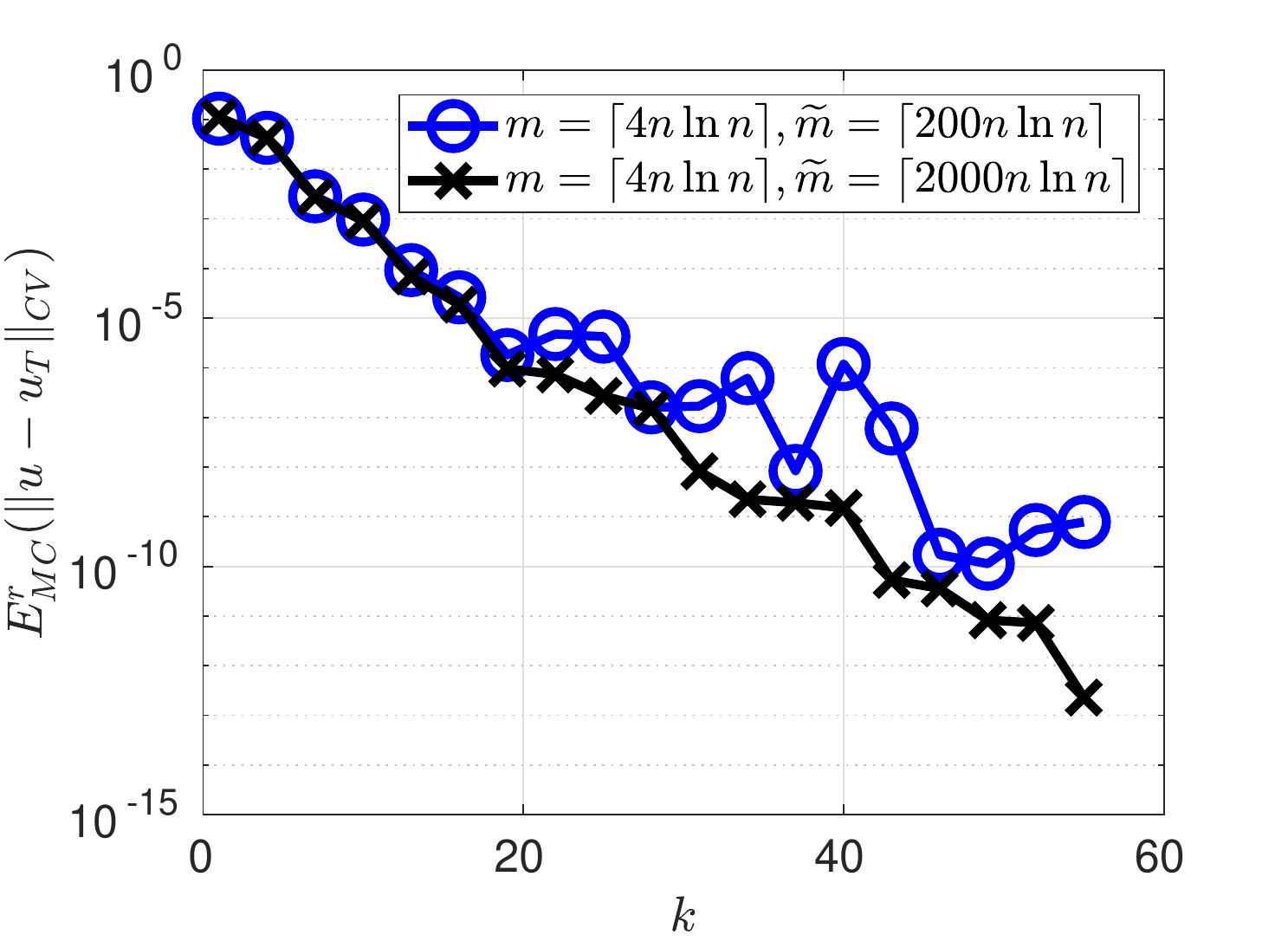} 
\includegraphics[height=6.5 cm ,clip]{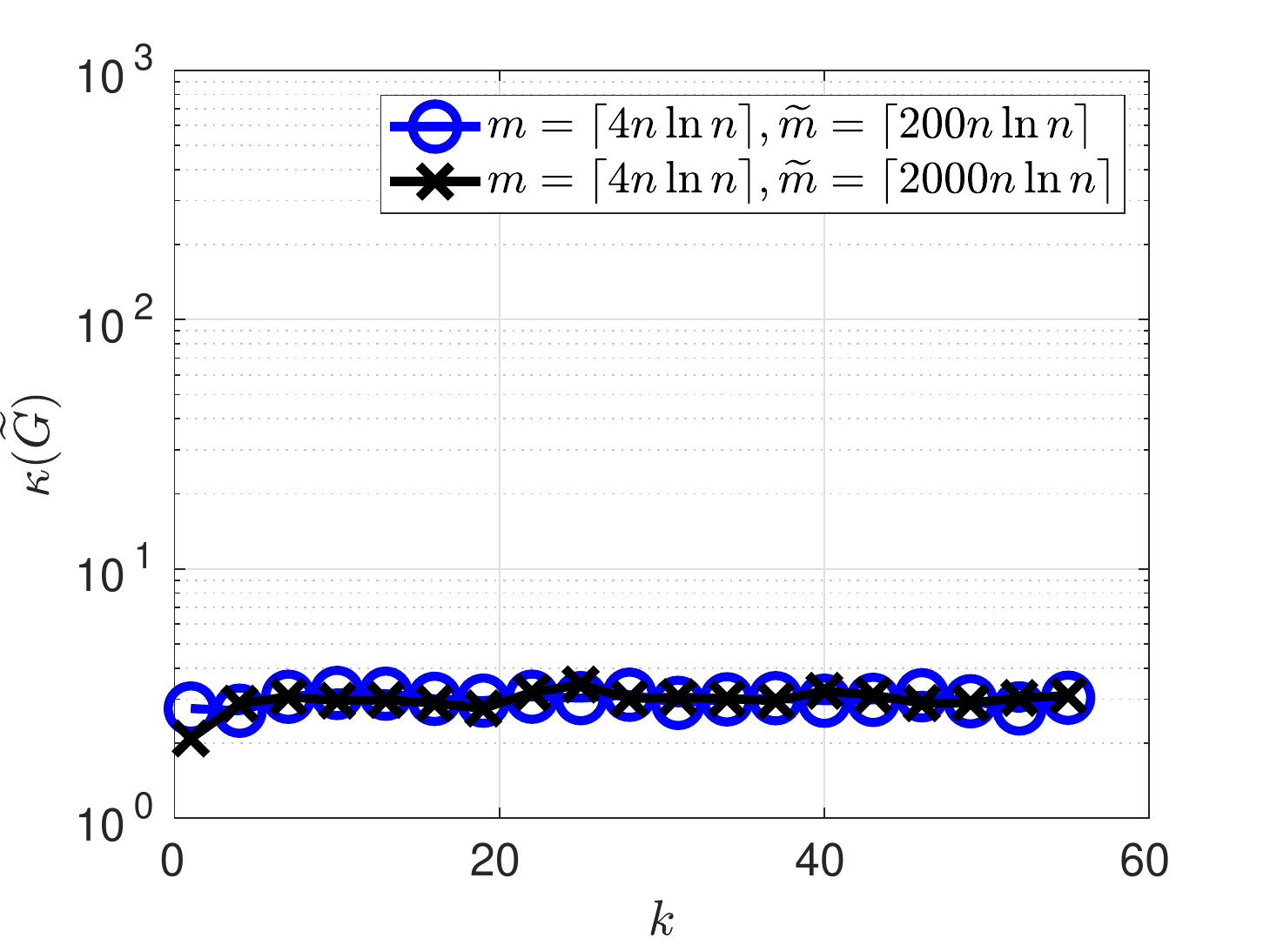} 
\caption{\label{fig:mandel} 
Left: error $\mathbb{E}_{MC}^r(\|\fun - \fun_T \|_{CV})$ for the function \eqref{eq:fun_cossin}. 
Right: condition number $\kappa(\widetilde{G})$.  
}
\end{figure}

\subsection{Comparison with examples from the literature} 
The following two examples are taken from \cite{AH}. Consider the function 
\begin{equation}
\fun(y)= \cos(2y_1) \sin(y_2), \quad y=(y_1,y_2)^\top \in \dom \subset \mathbb{R}^2,
\label{eq:fun_cossin}
\end{equation}
when $\dom$ is the Mandelbrot set displayed in Figure~\ref{fig:mandel_log}-right, or the function 
\begin{equation}
\fun(y)=  \sum_{j=1}^2 \left| y_j \right|^{-1/2}, \quad y=(y_1,y_2)^\top \in \dom \subset \mathbb{R}^2,
\label{eq:fun_sing}
\end{equation}
when $\dom=\{ y \in \alldom \, : \, \frac14  \leq \| y \|_{\ell^2} \leq 1 \}$ is the annular set displayed in Figure~\ref{fig:annulus}-right.  
With both functions, the space $V_n$ is chosen as the polynomial space supported on the hyperbolic cross index set of order $k$ defined as $\Lambda=\Lambda_{HC}^{d,k}:=\{ \nu \in \mathbb{N}_0^d \, : \, \prod_{j=1}^d (\nu_j+1)  \leq k+1\}$.

Figure~\ref{fig:mandel} shows the error and condition number for the example with the function \eqref{eq:fun_cossin} on the Mandelbrot set. 
When choosing $m=\lceil 4n\ln n \rceil$ and $\widetilde{m} = \lceil 200 n \ln n \rceil$, the error in Figure~\ref{fig:mandel} decreases exponentially w.r.t~$k$ up to $k=19$, and then exhibits an increasing variability and suboptimal convergence rate for $k> 19$. This is due to an underestimation of $K_n(\dom)$ when choosing $\costi=200$ for the given domain. Taking a larger $\costi=2000$ restores the exponential convergence of the error, at least for $k$ up to $57$.   
In Figure~\ref{fig:mandel_log}-left we report the same results as Figure~\ref{fig:mandel}-left but with $n$ in abscissa. 
Figure~\ref{fig:mandel_log}-right shows one realization of the pointwise error $y\mapsto \log_{10}|\fun(y)-\fun_T(y)|$ on $\dom$, obtained from the simulation in Figure~\ref{fig:mandel_log}-left when $n=176$, and the maximum error over $\dom$ is of the order $10^{-8}$.

\begin{figure}[htbp]   
\includegraphics[height=6.5 cm ,clip]{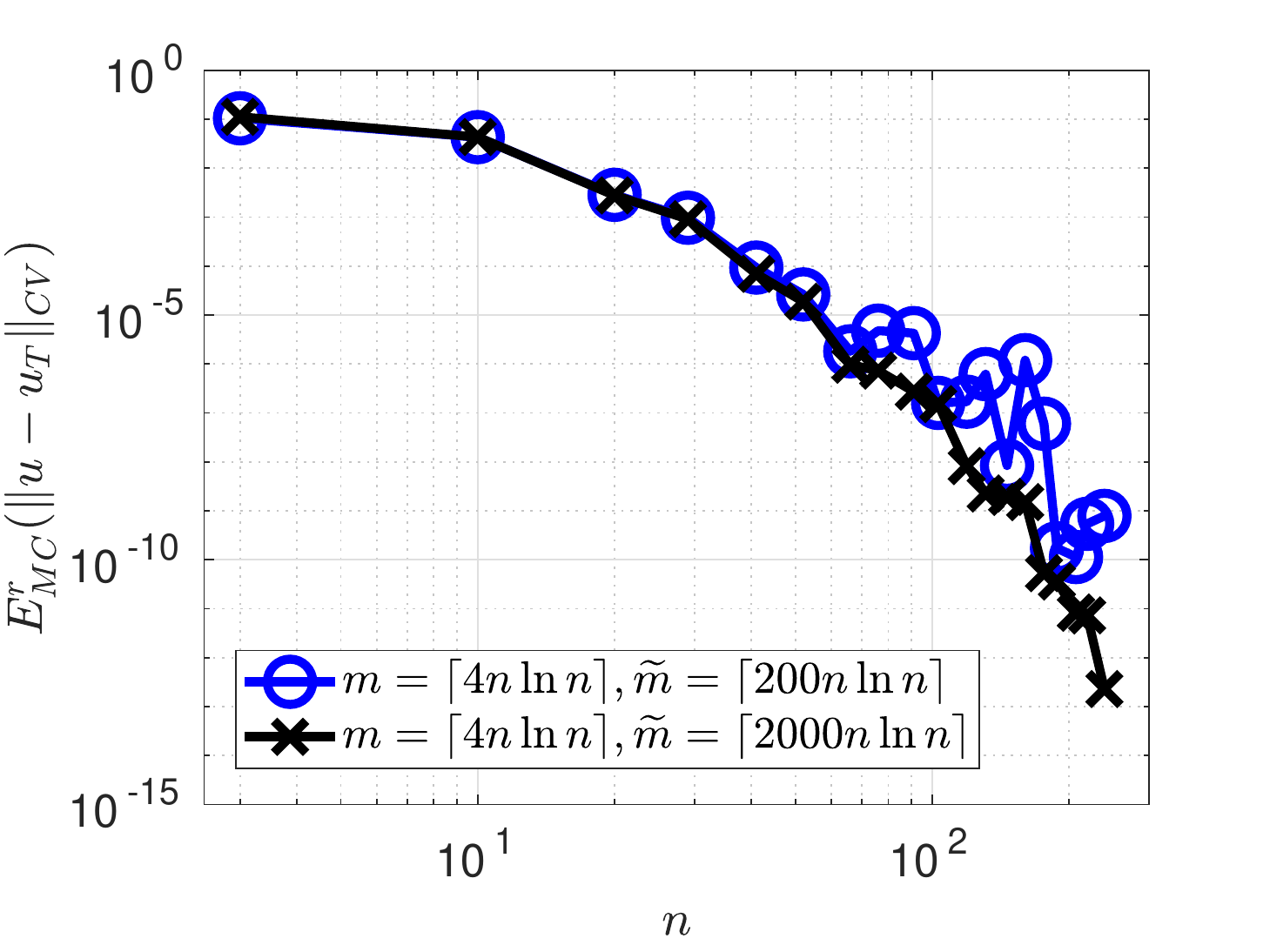} 
\includegraphics[height=6.5 cm ,clip]{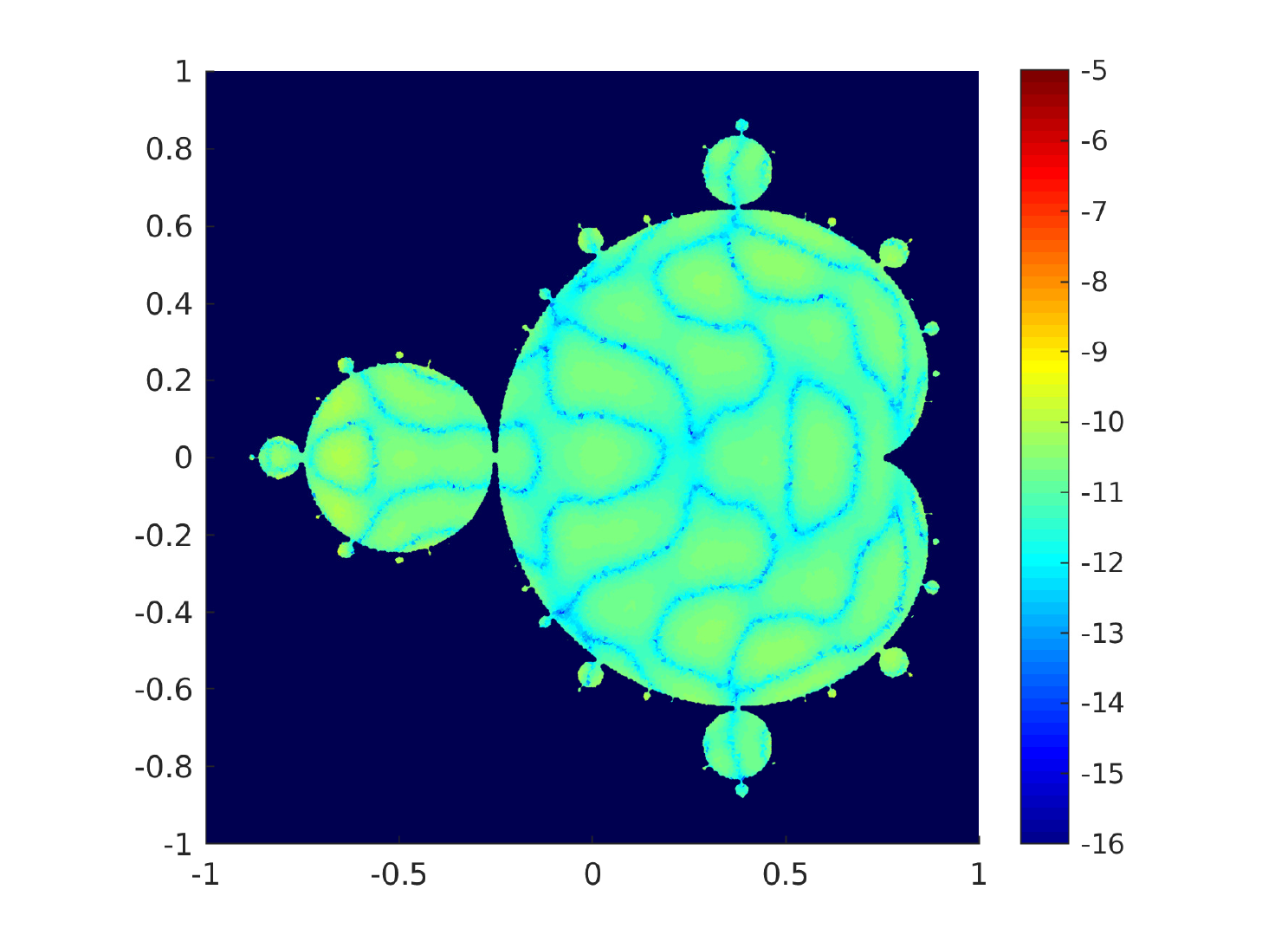} 
\caption{\label{fig:mandel_log} 
Left: same as Figure~\ref{fig:mandel}-left but with $n$ in abscissa.
Right: one realization of the error $y\mapsto \log_{10} |\fun(y) - \fun_T(y)|$ when $\dom$ is the Mandelbrot set, $\fun$ is the function \eqref{eq:fun_cossin} and $n=176$. 
The dark blue region corresponds to $\alldom\setminus\dom$. 
}
\end{figure}

Figure~\ref{fig:annulus}-left shows the error for the example with the nonsmooth function \eqref{eq:fun_sing} on the annular set, with $m=\lceil 4n\ln n \rceil$ and $\widetilde{m} = \lceil 200 n \ln n \rceil$.   
The corresponding results for the condition number are the same as the blue data in Figure~\ref{fig:mandel}-right, since both examples use the same polynomial space.
The error in Figure~\ref{fig:annulus}-left decreases algebraically w.r.t.~$n$.
One realization of the error is shown in Figure~\ref{fig:annulus}-right: the maximum error over $\dom$ equals $0.45$ and is attained along the discontinuities of $\fun$ on the Cartesian axes. 
The error in Figure~\ref{fig:annulus}-left does not manifest any instability, in contrast to the error in \cite[Figure 5]{AH} obtained for the same testcase but with the different method there proposed. 
In general, the error of the estimator $\fun_T$ is not affected by the distance of $\dom$ from the boundary of $\alldom$, even when $\dom$ touches $\partial\alldom$, like in this example.

\section{Conclusions and perspectives}
\label{sec:conclu}
We have developed and analysed numerical algorithms for the construction of weighted least-squares estimators in any $n$-dimensional space $V_n \subset L^2(\dom,\mu)$ defined 
on a general bounded domain $\dom$, when an explicit $L^2(\dom,\mu)$-orthonormal basis is not available.   
The estimator is stable with high probability, quasi-optimally converging in expectation, and uses a number of function evaluations $m$ of the order $n \ln n$. The calculation of the estimator requires the numerical construction of a discretely orthonormal surrogate basis $\widetilde{L}_1,\ldots,\widetilde{L}_n$ of $V_n$, at a computational cost that depends on the Christoffel function of $\dom$ and $V_n$.

The results in 
Theorem~\ref{main_theo}
apply 
to any general orthonormalisation algorithm that 
can construct an $\ortopar$-orthonormal surrogate basis $\widetilde{L}_1,\ldots,\widetilde{L}_n$ for $V_n$ with some probability $1-\conf$.
When using the Householder QR factorisation and $V_n$ is a multivariate polynomial space,  
$\ortopar$ is provably tiny 
for $n$ up to thousands, and $\conf=0$. 

An important point in the numerical construction of the surrogate basis is the robustness to ill-conditioning arising from the lack of knowledge of an $L^2(\dom,\mu)$-orthonormal basis. 
The algorithms proposed in this paper are extremely robust to such an ill-conditioning, and compute weighted least-squares estimators that are numerically stable and accurate with all the functions and domains tested.  

As a final remark, 
the whole analysis in this paper immediately applies to the adaptive setting,  
using nested sequences of approximation spaces $V^1 \subset \cdots \subset V^k \subset L^2(\dom,\mu)$ 
rather than a single a priori given approximation space.

\begin{figure}[htbp]   
\includegraphics[height=6.5 cm ,clip]{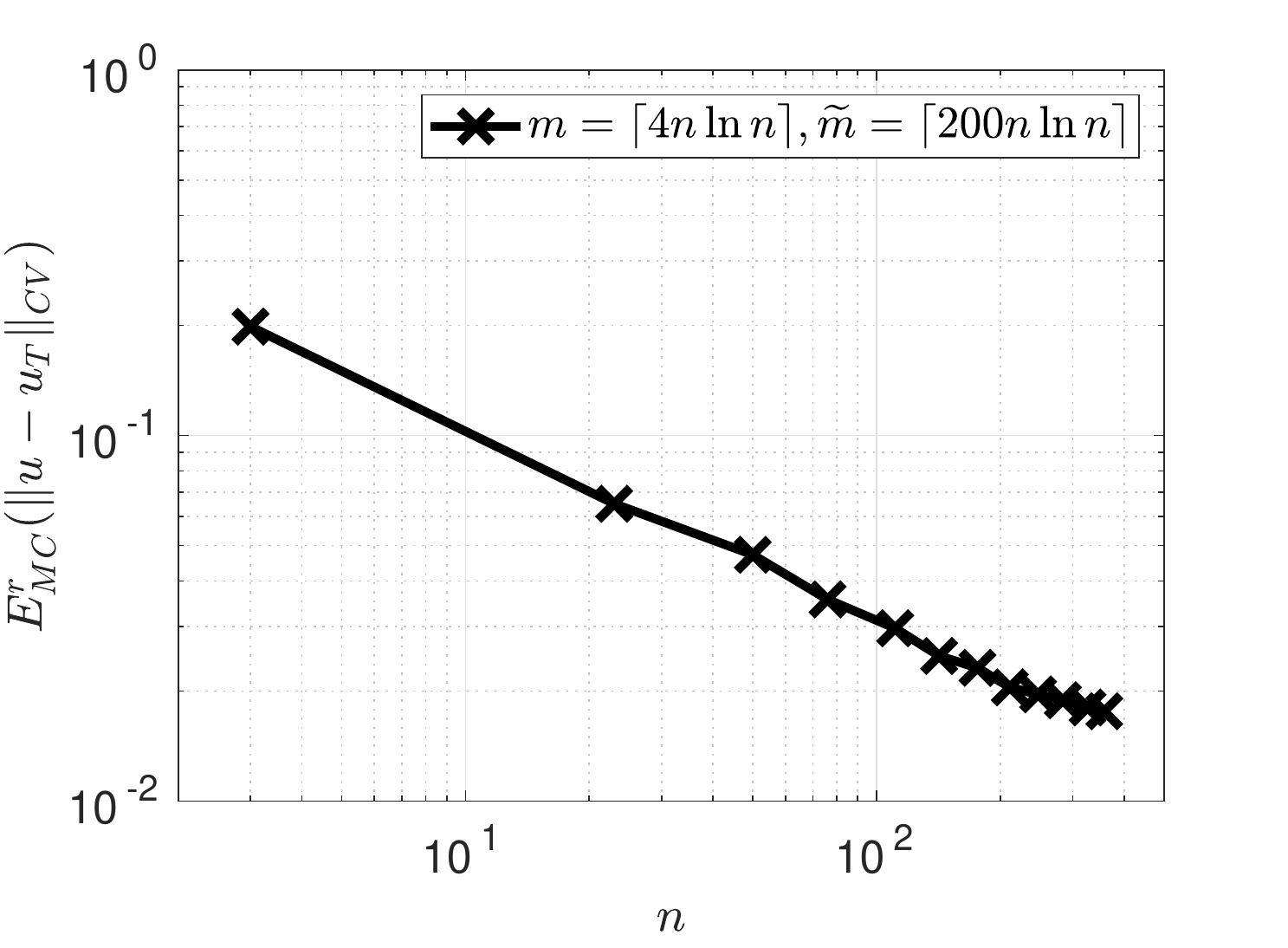}
\includegraphics[height=6.5 cm ,clip]{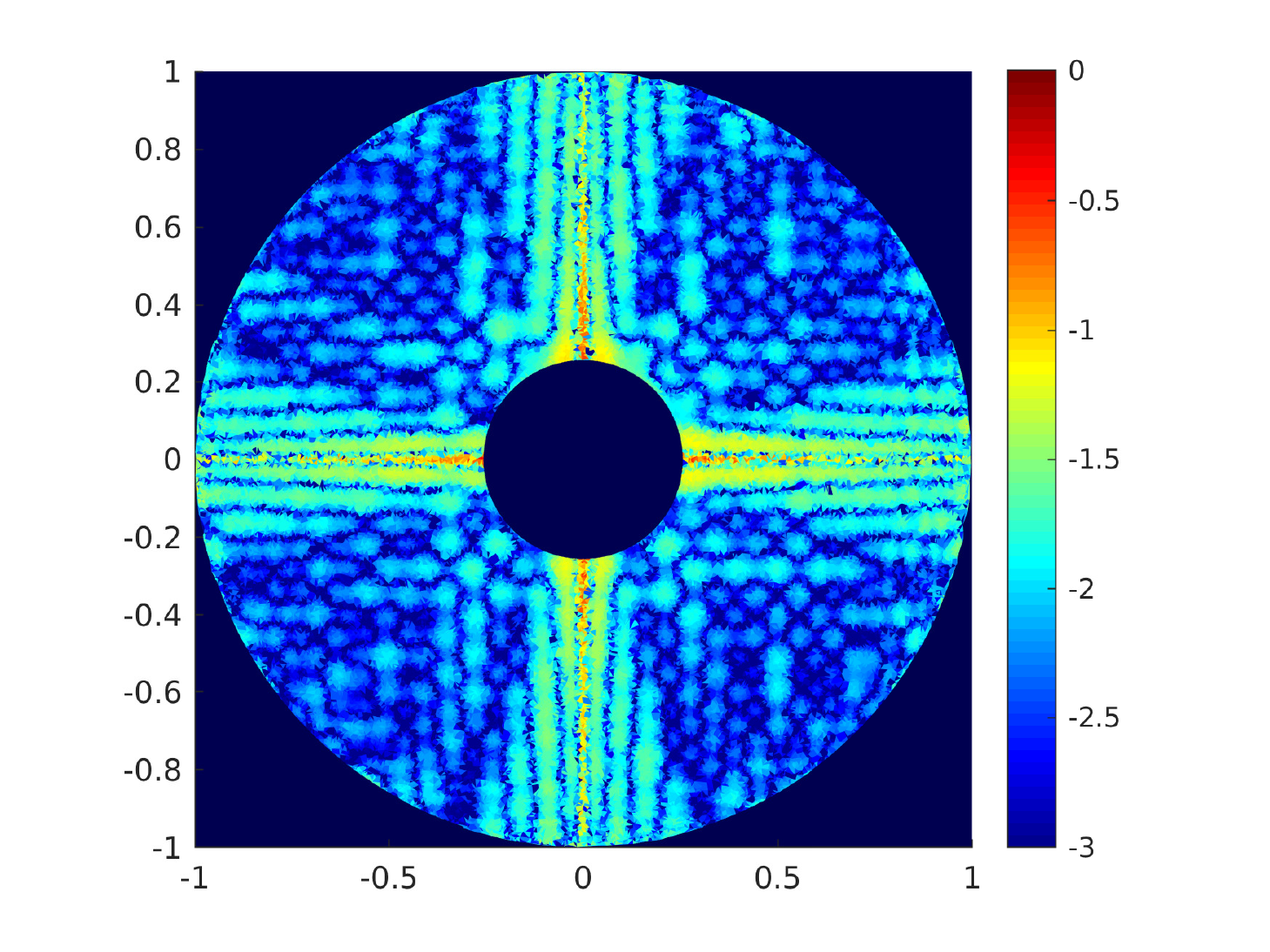}
\caption{\label{fig:annulus} 
Left: error $\mathbb{E}_{MC}^r(\|\fun - \fun_T \|_{CV})$ for the function \eqref{eq:fun_sing}.
Right: one realization of the error $y\mapsto \log_{10} |\fun(y) - \fun_T(y)|$ when $\dom$ is the annular set, $\fun$ is the function \eqref{eq:fun_sing} and $n=358$. 
The dark blue region corresponds to $\alldom\setminus\dom$. 
}
\end{figure}

\bibliographystyle{plain}

\end{document}